\documentclass{amsart}
\usepackage{amsmath,amsthm}
\usepackage{amsfonts,amssymb}

\newtheorem{thm}{Theorem}[section]
\newtheorem{cor}[thm]{Corollary}
\newtheorem{lem}[thm]{Lemma}
\newtheorem{prop}[thm]{Proposition}

\theoremstyle{remark}

 \def\t{{\theta}}
 \def\l{{\lambda}}
 \def\d{{\delta}}
  \def\eps{{\varepsilon}}

 \def\CG{{\mathcal G}}

 \def\la{{\langle}}
 \def\ra{{\rangle}}
 \def\ve{{\varepsilon}}

 \def\jb{{\mathbf j}}
 \def\kb{{\mathbf k}}

 \def\tb{{\mathbf t}}

 \def\CS{{\mathcal S}}
 
 \def\CI{{\mathcal I}}
 
  \def\CK{{\mathcal K}}
 \def\CL{{\mathcal L}}
 \def\CO{{\mathcal O}}
 \def\CP{{\mathcal P}}
 
 \def\CT{{\mathcal T}}

 \def\B{{\mathbb B}}
  
  \def\HH{{\mathbb H}}
    \def\GG{{\mathbb G}}
  
 \def\NN{{\mathbb N}}

 \def\RR{{\mathbb R}}
 
 \def\ZZ{{\mathbb Z}}

 \newcommand{\e}{\mathrm{e}}
 \newcommand{\tr}{{\mathsf {tr}}}

\newcommand{\wt}{\widetilde}

 \usepackage{graphicx}
 \usepackage{ifpdf}
  \ifpdf
 \DeclareGraphicsExtensions{.pdf,.png,.jpg}
 \else
 \DeclareGraphicsExtensions{.eps}
 \fi

\begin{document}

\title[]
{Cubature formula and interpolation on the cubic domain}

\author{Huiyuan Li}
\address{Institute of Software\\
Chinese Academy of Sciences\\ Beijing 100080,China}
\email{hynli@mail.rdcps.ac.cn}
\author{Jiachang Sun}
\address{Institute of Software\\
Chinese Academy of Sciences\\ Beijing 100080,China}
\email{sun@mail.rdcps.ac.cn}
\author{Yuan Xu}
\address{  Department of Mathematics\\ University of Oregon\\
    Eugene, Oregon 97403-1222.}
\email{yuan@math.uoregon.edu}

\date{\today}
\thanks{The first and the second authors were supported by NSFC Grant
10431050 and 60573023. The second author was supported by National Basic
Research Program grant 2005CB321702. The third author was supported by
NSF Grant DMS-0604056}

\begin{abstract}
Several cubature formulas on the cubic domains are derived using the
discrete Fourier analysis associated with lattice tiling, as developed in
\cite{LSX}. The main results consist of a new derivation of the Gaussian type
cubature for the product Chebyshev weight functions and associated
interpolation polynomials on $[-1,1]^2$, as well as new results on $[-1,1]^3$.
In particular, compact formulas for the fundamental interpolation polynomials
are derived, based on $n^3/4 +\CO(n^2)$ nodes of a cubature formula
on $[-1,1]^3$.
\end{abstract}

\maketitle

\section{Introduction}
\setcounter{equation}{0}

For a given weight function $W$ supported on a set $\Omega \in \RR^d$, a
cubature formula of degree $2n-1$ is a finite sum, $L_n f$,  that provides an
approximation to the integral and preserves polynomials of degree up to $2n-1$;
that is,
$$
 \int_{\Omega} f(x) W(x) dx = \sum_{k=1}^N \lambda_k f(x_k)=: L_n f  \qquad
      \hbox{for all $f \in \Pi_{2n-1}^d$},
$$
where $\Pi_M^d$ denotes the space of polynomials of total degree at most
$n$ in $d$ variables. The points $x_k \in \RR^d$ are called {\it nodes} and
the numbers $\lambda_k \in \RR \setminus \{0\}$ are called {\it weights} of
the cubature.

Our primary interests are Gaussian type cubature,  which has minimal or
nearer minimal number of nodes. For $d =1$, it is well known that Gaussian
quadrature of degree $2n-1$ needs merely $N =n$ nodes and these nodes
are precisely the zeros of the orthogonal polynomial of degree $n$ with
respect to $W$. The situation for $d \ge 1$, however, is much more complicated
and not well understood in general.  As in the case of $d =1$, it is known that a
cubature of degree $2n-1$ needs at least $N \ge \dim \Pi_{n-1}^d$ number of
nodes, but few formulas are known to attain this lower bound (see, for
example, \cite{BSX, LSX}). In fact, for the centrally symmetric weight function
(symmetric with respect to the origin), it is known that the number of nodes,
$N$, of a cubature of degree $2n-1$ in two dimension satisfies the lower bound
\begin{equation} \label{Moller}
   N \ge \dim \Pi_{n-1}^2 + \left \lfloor \frac{n}2 \right \rfloor,
\end{equation}
known as M\"oller's lower bound \cite{Mo1}.   It is also known that the nodes
of a cubature that attains the lower bound \eqref{Moller}, if it exists, are
necessarily the common zeros
of $n+1 - \lfloor \frac{n}2 \rfloor$ orthogonal polynomials of degree $n$ with
respect to $W$. Similar statements on the nodes hold for cubature formulas
that have number of nodes slightly above M\"oller's lower bound, which we shall
call cubature of {\it Gaussian type}.  These definitions also hold in $d$-dimension,
where the lower bound for the number of nodes for the centrally symmetric
weight function is given in \cite{Mo2}.

There are, however, only a few examples of such formulas that are explicitly
constructed and fewer still  can be useful for practical computation. The
best known example is $\Omega = [-1,1]^d$ with the weight function
\begin{equation}\label{W0W1}
  W_0(x):= \prod_{i=1}^d  \frac{1}{\sqrt{1-x_i^2}} \qquad \hbox{or}\qquad
      W_1(x):= \prod_{i=1}^d  \sqrt{1-x_i^2}
\end{equation}
and only when $d =2$. In this case, several families of Gaussian type cubature
are explicitly known, they were constructed (\cite{MP, X94}) by studying the
common zeros of corresponding orthogonal polynomials, which are product
Chebyshev polynomials of the first kind and the second kind, respectively.
Furthermore, interpolation polynomials bases on the nodes of these cubature
formulas turn out to possess several desirable features (\cite{X95}, and
also \cite{Xpt3}). On the other hand,  studying common zeros of orthogonal
polynomials of several variables is in general notoriously difficult. In the case
of \eqref{W0W1}, the product Chebyshev polynomials have the simplest
structure among all orthogonal polynomials, which permits us to study their
common zeros and construct cubature formulas in the case $d=2$, but not
yet for the case $d=3$ or higher.

The purpose of the present paper is to provide a completely different method
for constructing cubature formulas with respect to $W_0$ and $W_1$. It uses
the discrete Fourier analysis associated with lattice tiling, developed recently
in \cite{LSX}. This method has been used in \cite{LSX} to establish cubature
for {\it trigonometric functions} on the regular hexagon and triangle in $\RR^2$,
a topic that has been studied in \cite{Sun, LS}, and on the rhombic dodecahedron
and tetrahedron of $\RR^3$ in \cite{LX}. The cubature on the hexagon can be
transformed, by symmetry, to a cubature on the equilateral triangle that
generates the hexagon by reflection, which can in turn be further transformed,
by a nontrivial change of variables, to Gaussian cubature formula for algebraic
polynomials on the domain bounded by Steiner's hypercycloid. The theory
developed in \cite{LSX} uses two lattices, one determines the domain of
integral and the points that defined the discrete inner product, the other
determines the space of exponentials or trigonometric functions that are
integrated exactly by the cubature. In \cite{LSX,LX} the two lattices are taken
to be the same. In this paper we shall choose one as $\ZZ^d$ itself, so that
the integral domain is fixed as the cube, while we choose the other one
differently. In $d =2$, we choose the second lattice so that its spectral set
is a rhombus, which allows us to establish cubature formulas for trigonometric
functions that are equivalent to Gaussian type cubature formulas for $W_0$
and $W_1$. In the case of $d =3$, we choose the rhombic dodecahedron as
a tiling set and obtain a cubature of degree $2n-1$ that uses $n^3/4 + \CO(n^2)$
nodes, worse than the expected lower bound of $n^3/6+ \CO(n^2)$ but far
better than the product Gaussian cubature of $n^3$ nodes. This cubature
with $n^3/4 + \CO(n^2)$ nodes has appeared recently and tested numerically
in \cite{MVX}. We will further study the Lagrange interpolation based on its
nodes, for which the first task is to identify the subspace that the interpolation
polynomials belongs. We will not only identify the interpolation space, but
also give the compact formulas for the fundamental interpolation polynomials.

One immediate question arising from this study is if there exist cubature
formulas of degree $2n-1$ with $n^3/6+\CO(n^2)$ nodes on the cube.
Although examples of cubature formulas of degree $2n-1$ with $N =
\dim \Pi_{n-1}^d = n^d/d! + \CO(n^{d-1})$ nodes are known to exist for
special non-centrally symmetric regions (\cite{BSX}), we are not aware of
any examples for symmetric domains that use $N= n^d/d! + \CO(n^{d-1})$
nodes. From our approach of tiling and discrete Fourier analysis, it appears
that the rhombic dodecahedron gives the smallest number of nodes among
all other fundamental domains that tile $\RR^3$ by translation. Giving the
fact that this approach yields the cubature formulas with optimal order
for the number of nodes, it is tempting to make the conjecture that a cubature
formula  of degree $2n-1$ on $[-1,1]^3$ needs at least $n^3/4 + \CO(n^2)$
nodes.

The paper is organized as follows. In the following section we recall the
result on discrete Fourier analysis and lattice tiling in \cite{LSX}. Cubature
and interpolation for $d =2$ are developed in Section 3 and those for $d=3$
are discussed in Section 4, both the latter two sections are divided into
several subsections.

\section{Discrete Fourier Analysis with lattice Tiling}
\setcounter{equation}{0}

We recall basic results in \cite{LSX} on the discrete Fourier analysis associated
with a lattice. A lattice of $\RR^d$ is a discrete subgroup that can be written
as $A \ZZ^d =\{Ak: k \in \ZZ^d\}$, where $A$ is a $d\times d$ invertible matrix,
called the generator of the lattice.  A bounded set $\Omega_A \subset \RR^d$
is said to tile $\RR^d$ with the lattice $A\ZZ^d$ if
$$
   \sum_{k \in \ZZ^d} \chi_{\Omega_A} (x+ A k) =1 \qquad \hbox{for almost
         all $x \in \RR^d$},
$$
where $\chi_{_E}$ denotes the characteristic function of the set $E$. 
The simplest lattice is $\ZZ^d$ itself, for which the set that tiles $\RR^d$ is
$$
        \Omega := [-\tfrac12, \tfrac12)^d.
$$
We reserve the notation $\Omega$ as above throughout the rest of this paper.
The set $\Omega$ is chosen as half open so that its translations by $\ZZ^d$ tile
$\RR^d$ without overlapping. It is well known that the exponential functions
$$
   \e_k(x) := \e^{2\pi i k \cdot x}, \qquad k \in \ZZ^d, \quad x \in \RR^d,
$$
form an orthonormal basis for $L^2(\Omega)$. These functions are periodic
with respect to $\ZZ^d$; that is, they satisfy
$$
    f(x + k) = f (x) \qquad \hbox{for all $k\in \ZZ^d$}.
$$
Let $B$ be a $d\times d$ matrix such that all entries of $B$ are integers.
Denote
\begin{equation}\label{LambdaB}
        \Lambda_{B} = \left \{k \in \ZZ^d: B^{-\tr} k \in  \Omega \right\}
                                                  \quad \hbox{and} \quad
        \Lambda_B^\dag = \left \{k \in \ZZ^d: k \in \Omega_B \right\}.
\end{equation}
It is known that $|\Lambda_{B}| = |\Lambda^\dag_B|= |\det B|$, where
$|E|$ denotes the cardinality of the set $E$.  We need the following
theorem \cite[Theorem 2.5]{LSX}.

\begin{thm} \label{thm:2.1}
Let $B$ be a $d \times d$ matrix with integer entries. Define the discrete inner
product
\begin{align*}
 \langle f,g \rangle_B := \frac{1}{|\det(B)|} \sum_{j\in \Lambda_B} f(B^{-\tr}j )
    \overline{ g(B^{-\tr}j)}
\end{align*}
for $f,\, g \in C(\Omega)$, the space of continuous functions on $\Omega$.
Then
\begin{align}
\label{eq:equiv_inner}
\langle f,\,g \rangle_B = \langle f,\,g \rangle: = \int_{\Omega} f(x) \overline{g(x)}dx,
\end{align}
for all $f,\,g$ in the finite dimensional subspace
\begin{align*}
  \mathcal{T}_B := \mathrm{span}\left\{\e^{2\pi i\, k\cdot x}:
                 k \in \Lambda_B^\dag \right\}.
\end{align*}
The dimension of $\mathcal{T}_B$ is $|\Lambda_B^\dag|= |\det B|$.
\end{thm}

This result is a special case of a general result in \cite{LSX}, in which $\Omega$
is replaced by $\Omega_A$ for an invertible matrix $A$, and the set $\Lambda_B$
is replaced by $\Lambda_N$ with $N = B^\tr A$ and $N$ is
assumed to have integer entries. Since we are interested only at the cube
$[-\frac12,\frac12]^d$ in this paper, we have chosen $A$ as the identity matrix.

We can also use the discrete Fourier analysis to study interpolation based on
the points in $\Lambda_B$.  We say two points $x, y \in \RR^d$ congruent with
respect to the lattice $B\ZZ^d$, if $x - y \in B\ZZ^d$, and we write $x \equiv y
\mod B$. We then have the following result:

\begin{thm} \label{thm:interpolation}
For a generic function $f$ defined in $C(\Omega)$, the unique interpolation
function $\CI_B f$ in $\CT_B$ that satisfies
$$
    \CI_B f (B^{-\tr}j ) = f (B^{-\tr}j),\qquad  \forall j \in \Lambda_B
$$
is given by
\begin{align} \label{interpolation}
 \CI_B f (x)  =   \sum_{k \in \Lambda_B^\dag} \langle f, \e_k \rangle \e_k(x)
 = \sum_{k \in \Lambda_B} f(B^{-\tr}k)    \Psi_{\Omega_B}(x-B^{-\tr}k),
\end{align}
where
\begin{equation} \label{ell}
  \Psi_{\Omega_B} (x) = \frac{1}{|\det (B)|} \sum_{j \in \Lambda_{B^\dag}}
         \e^{2 \pi i  j^\tr  x}.
\end{equation}
\end{thm}

The proof of this result is based on the second one of the following two relations
that are of independent interests:
\begin{equation}\label{d-ortho1}
  \frac{1}{|\det(B)|}\sum_{j \in \Lambda_B} \e^{2 \pi i k^\tr B^{-\tr} j}
      = \begin{cases} 1,  &  \hbox{if  $k \equiv 0 \mod B$}, \\
              0, & \hbox{otherwise}, \end{cases}
\end{equation}
and
\begin{equation}\label{d-ortho2}
  \frac{1}{|\det(B)|}\sum_{k \in \Lambda_B^\dag} \e^{- 2 \pi i k^\tr B^{-\tr} j}
      = \begin{cases} 1,  &  \hbox{if  $j \equiv 0 \mod B^\tr$}, \\
              0, & \hbox{otherwise}. \end{cases}
\end{equation}

For proofs and further results we refer to \cite{LSX,LX}. Throughout this
paper we will write, for $k \in \ZZ^d$, $2 k = (2k_1,\ldots,2k_d)$ and
$2 k +1 = (2k_1+1,\ldots,2k_d+1)$.


\section{Cubature and Interpolation on the square}
\setcounter{equation}{0}

In this section we consider the case $d =2$. In the first subsection,
the general results in the previous section is specialized to a special case
and cubature formulas are derived for a class of trigonometric functions.
These results are converted to results for algebraic polynomials in the
second subsection. Results on polynomial interpolation are derived in the
third subsection.

\subsection{Discrete Fourier analysis and cubature formulas on the plane}
We choose the matrix $B$ as
$$
  B =  n  \left[ \begin{matrix} 1 & 1\\ -1 & 1 \end{matrix} \right]
    \qquad \hbox{and} \qquad
  B^{-1} =   \frac{1}{2n} \left[ \begin{matrix} 1 & - 1\\ 1 & 1 \end{matrix} \right].
$$
Since $B$ is a rotation, by 45 degree, of a constant multiple of the diagonal
matrix, it is easy to see that the domain $\Omega_B$ is defined by
$$
   \Omega_B = \{x \in \RR^2 : -n \le x_1+x_2 < n, \,\, -n \le x_2 - x_1 < n\},
$$
which is depicted in Figure 1 below.

\begin{figure}[ht]
\centering
\includegraphics[width=0.4\textwidth]{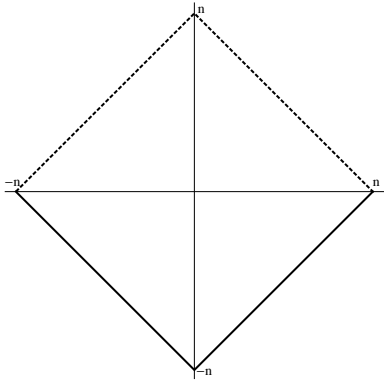}
\caption{Rhombus $\Omega_B$}
\end{figure}

\noindent
From the expression of $B^{-\tr}$, it follows readily that
$\Lambda_B = \Lambda_B^\dag  = :\Lambda_n$, where
$$
   \Lambda_n  =
       \{j \in \ZZ^2:  -n \le j_1+j_2 <n, \,\, -n \le j_2 - j_1 < n\}.
$$
The cardinality of $\Lambda_n$ is $|\Lambda_n| = 2 n^2$.
We further denote the space $\CT_B$ by $\CT_n$, which is given by
$$
  \CT_n : = \mathrm{span}
      \left\{\e^{2\pi i\, k\cdot x} : \ k \in \Lambda_{n} \right\}.
$$

\begin{thm} \label{inner1}
Define the set
$$
X_n: = \left \{  2 k : - \tfrac{n}{2}  \le k_1, k_2 < \tfrac{n}{2}\}   \cup  \{  2 k +1 :
       - \tfrac{n+1}{2}  \le k_1, k_2 < \tfrac{n-1}{2} \right\}.
$$
Then for all $f,g \in \CT_n$,
$$
  \langle f, g \rangle_n : = \frac{1}{2n^2} \sum_{k\in X_n} f(\tfrac{k}{2n})
         \overline{g(\tfrac{k}{2n})}
     = \int_{[-\frac12,\frac12]^2} f(x)\overline{g(x)} dx.
$$
\end{thm}

\begin{proof}
Changing variables from $j$ to $k = 2n B^{-\tr}j$, or $k_1 =j_1+j_2$
and $k_2 = j_2 - j_1$, then, as $j_1$ and $j_2$ need to be integers
and $j_1 = \frac{k_1- k_2}{2}$, $j_2 = \frac{k_1+k_2}{2}$, we see
that
\begin{equation} \label{j-k}
    j\in \Lambda_n  \qquad \Longleftrightarrow  \qquad k= 2n B^{-\tr}j \in X_n.
\end{equation}
Hence, as $\det(B) = 2n^2$, we conclude that $\la f,g\ra_n = \la f,
g \ra_B$ and this theorem follows as a special case of Theorem
\ref{thm:2.1}.
\end{proof}

The set $\Lambda_n$ lacks symmetry as the inequalities in its definition
are half open and half closed. We denote its symmetric counterpart by
$\Lambda_n^*$, which is defined by
$$
  \Lambda_n^* : = \{j \in \ZZ^2:  -n \le j_1+j_2 \le n, \,\, -n \le j_1 - j_2 \le n\}.
$$
We also denote the counterpart of $\CT_n$ by $\CT_n^*$, which is defined by
$$
  \CT_n^* : = \mathrm{span}
      \left\{\e^{2\pi i\, k\cdot x} : \ k \in \Lambda_{n}^* \right\}.
$$
Along the same line, we also define the counterpart of $X_n$ as
$$
X_n^* : = \left \{ 2 k: - \tfrac{n}{2} \le k_1, k_2 \le \tfrac{n}{2}\}   \cup  \{ 2 k+1:
      - \tfrac{n+1}{2}  \le k_1, k_2 \le \tfrac{n-1}{2}\right\}\!.
$$
It is easy to see that $|X_n| = |\Lambda_n| = 2n^2$, whereas $|X_n^*| =
2n^2 + 2n+1$. We further partition the set $X_n^*$ into three parts,
$$
   X_n^* = X_n^\circ  \cup X_n^e \cup X_n^v,
$$
where $X_n^{\circ} = X_n^* \cap (-n,n)^2$ is the set of interior points of
$X_n^*$,  $X_n^{e}$ consists of those points in $X_n^*$ that are on the
edges of $[-n,n]^2$ but not on the 4 vertices or corners, while $X_n^v$
consists of those points of $X_n^*$ at the vertices of  $[-n,n]^2$.

\begin{thm} \label{inner2}
Define the inner product
\begin{equation} \label{eq:inner2}
\la f, g \ra_n^* : = \frac{1}{2 n^2} \sum_{k \in X_n^*} c_k^{(n)} f(\tfrac{k}{2n})
          \overline{g(\tfrac{k}{2n})},
       \qquad \hbox{where}\quad
            c_k^{(n)} = \begin{cases} 1, & k \in X_n^\circ \\ \frac12, & k \in X_n^e \\
                                             \frac14, & k \in X_n^v   \end{cases}.
\end{equation}
Then for all $f,g \in \CT_n$,
$$
 \int_{[-\frac12,\frac12]^2} f(x)\overline{g(x)} dx = \langle f, g \rangle_n
        = \langle f, g \rangle_n^*.
$$
\end{thm}

\begin{proof}
Evidently we only need to show that $\langle f, g \rangle_n = \langle f, g \rangle_n^*$.
Since $c_k^{(n)} =1$ for $k \in X_n^\circ$, the partial sums over interior points
of the two sums agree. The set $X_n^e$ of boundary points can be divided into
two parts, $X_n^e = X_n^{e,1} \cup X_n^{e,2}$, where $X_n^{e,1}$ consists of
points in $X_n$ that are on the edges of $[-n, n)^2$, but not equal to $(-n,-n)$,
and $X_n^{e,2}$ is the complementary of $X_n^{e,1}$ in $X_n^e$. Evidently,
if $x \in X_n^{e,1}$, then either $x + (2n,0)$ or $x+(0,2n)$
belongs to $X_n^{e,2}$. Hence,  if $f$ is a periodic function, $f(x+k) = f(x)$
for $k \in \ZZ^2$, then
$$
\sum_{k \in X_n^e} c_k^{(n)} f(\tfrac{k}{2n}) =
  \frac12 \sum_{k \in X_n^{e}}  f(\tfrac{k}{2n}) =
        \sum_{k \in X_n^{e,1}} f(\tfrac{k}{2n}).
$$
Furthermore, for $(-n, -n) \in X_n$,  $X_n^*$ contains all four vertices
$(\pm n, \pm n)$. Since a periodic function takes the same value on all four
points, $\sum_{k \in X_n^v} c_k^{(n)} f(\tfrac{k}{2n}) = f(-\frac12, -\frac12)$.
Consequently, we have proved that $\langle f, g \rangle_n = \langle f, g\rangle_n^*$
if $f, g$ are periodic functions.
\end{proof}

As a consequence of the above two theorems, we deduce the following two
cubature formulas:

\begin{thm} \label{thm:cuba1}
For $n \ge 2$, the cubature formulas
\begin{equation} \label{cuba1}
   \int_{[-\frac12,\frac12]^2} f(x) dx =  \frac{1}{2n^2} \sum_{k \in X_n^*} c_k^{(n)} f(\tfrac{k}{2n})
     \quad\hbox{and} \quad
  \int_{[-\frac12,\frac12]^2} f(x) dx =  \frac{1}{2n^2} \sum_{k \in X_n} f(\tfrac{k}{2n})
\end{equation}
are exact for $f \in \CT_{2n-1}^*$.
\end{thm}

\begin{proof}
It suffices to proof that both cubature formulas in \eqref{cuba1} are exact for 
every $\e_j$ with $j\in \Lambda^*_{2n-1}$. For this purpose, we first claim that 
for any $j\in \ZZ^2$, there exist $\nu\in \Lambda_n$ and $l\in \ZZ^2$ such that
$j=\nu+Bl$. Indeed, the translations of $\Omega_B$ by $B\ZZ^2$ tile $\RR^2$, 
thus we have $j=x+Bl$ for certain $x\in \Omega_B$ and $l\in \ZZ^2$. Since 
all entries of the matrix $B$ are integers, we further deduce that 
$\nu:=x=j-Bl\in \ZZ^2\cap \Omega_B=\Lambda_n$.

Next assume $j\in \Lambda^*_{2n-1}$. Clearly  the integral of $\e_j$ over $\Omega$ is
$\delta_{j,0}$. On the other hand, let us suppose $j=\nu+Bl $ with $\nu\in \Lambda_n$ and $l\in \ZZ^2$. Then it is easy to see that $\e_j(\frac{k}{2n}) = \e_{\nu}(\frac{k}{2n})$
for each $k\in X^*_n$. Consequently, we obtain from Theorem \ref{inner2} that
\begin{align*}
   \sum_{k\in X_n^*} c_{k}^{(n)}\e_{j}(\tfrac{k}{2n}) & = \sum_{k\in X_n^*}
     c_{k}^{(n)}\e_{\nu}(\tfrac{k}{2n})  = \sum_{k\in X_n}
     \e_{\nu}(\tfrac{k}{2n}) \\ 
      & = \sum_{k\in X_n} \e_{j}(\tfrac{k}{2n})
     = \int_\Omega  \e_{\nu}(x) dx = \delta_{\nu,0}.
\end{align*}
Since  $\nu=0$ implies $j=Bl\in \ZZ^2$ which gives $j=l=0$, we further obtain 
that $\delta_{\nu,0}=\delta_{j,0}$. This completes the proof of \eqref{cuba1}.
\end{proof}

We note that the second cubature in \eqref{cuba1} is a so-called Chebyshev
cubature; that is, all its weights are equal.


\subsection{Cubature for algebraic polynomials}

The set $\Lambda_n^*$ is symmetric with respect to the mappings
$(x_1,x_2) \mapsto (-x_1,x_2)$ and $(x_1,x_2) \mapsto (x_1, -x_2)$.
It follows that both the spaces
\begin{align*}
 \CT_n^{\rm{even}}:& =\mathrm{span} \{\cos 2 \pi j_1x_1 \cos 2 \pi j_2 x_2:
     0 \le j_1 + j_2 \le n \}, \\
 \CT_n^{\rm{odd}}: & = \mathrm{span}
    \{\sin 2 \pi j_1x_1 \sin 2 \pi j_2 x_2:  1 \le j_1 + j_2 \le n \}
\end{align*}
are subspaces of $\CT_n^*$.  Recall that Chebyshev polynomials of the first kind,
$T_n(t)$, and the second kind, $U_n(t)$, are defined, respectively, by
$$
 T_n(t)= \cos  n \theta \quad\hbox{and}\quad
  U_n(t) = \frac{\sin (n+1)\t}{\sin \t},
   \qquad t = \cos  \t.
$$
They are orthogonal with respect to $w_0(t) =1/\sqrt{1-t^2}$ and $w_1(t)
= \sqrt{1-t^2}$ over $[-1,1]$, respectively. Both are algebraic polynomials of
degree $n$ in $t$. Recall the definition of $W_0$ and $W_1$ in \eqref{W0W1}.
Under the changing of variables
\begin{equation} \label{t-x}
 t_1 = \cos 2\pi x_1, \quad t_2 = \cos 2\pi x_2,  \qquad
     (x_1,x_2) \in [-\tfrac12, \tfrac12]^2,
\end{equation}
the subspace $\CT_n^{\rm even}$ becomes the space $\Pi_n^2$ of polynomials
of degree $n$ in the variables $(t_1,t_2)$,
$$
   \Pi_n^2  = \mathrm{span} \{T_j(t_1)T_{k-j}(t_2): 0 \le j \le k \le n\}
$$
and the orthogonality of $\e_k$ over $\Omega$ implies that $T_j^k(t):=
T_j(t_1)T_{k-j}(t_2)$ are orthogonal polynomials of two variables,
$$
   \frac{1}{\pi^2} \int_{[-1,1]^2} T_j^k(t) T_{j'}^{k'}(t)  W_0(t) dt
   =\begin{cases}  1,& k=k'=j=j'=0,\\
     \tfrac12, & (k,j) = (k',j') \hbox{ and }  (k-j)j=0,\\
     \tfrac14, & k=k'>j=j'>0,\\
      0, &(k,j) \ne (k',j').
      \end{cases}
$$
We note also that the subspace $\CT_n^{\mathrm{odd}}$ becomes the space
$\{\sqrt{1-t_1^2} \sqrt{1-t_2^2}\, p(t): p \in \Pi_{n-1}^2\}$ in the variables
$t=(t_1,t_2)$, and the orthogonality of $\e_k$ also implies that
$U_j^k(t):=
U_j(t_1)U_{k-j}(t_2)$ are orthogonal polynomials of two variables,
$$
   \frac{1}{\pi^2} \int_{[-1,1]^2} U_j^k(t) U_{j'}^{k'}(t)
    W_1(t) dt  = \frac{1}{4} \delta_{j,j'}\delta_{k,k'}.
$$

The symmetry allows us to translate the results in the previous
subsection to algebraic polynomials. Since $\cos 2\pi j_1 x_1 \cos
2\pi j_2 x_2$ are even in both variables, we only need to consider
their values over $X_n^* \cap \{x: x_1 \ge 0, x_2 \ge 0\}$. Hence,
we define
\begin{equation} \label{Xi}
  \Xi_n: = \{(2k_1,{2 k_2}): 0 \le k_1,k_2 \le \tfrac{n}{2}\}
       \cup  \{({2k_1+1},{2k_2+1}): 0 \le k_1,k_2 \le \tfrac{n-1}{2}\},
\end{equation}
and, under the change of variables \eqref{t-x},
\begin{equation} \label{G_n}
  \Gamma_n: = \{(z_{k_1},z_{k_2}): (k_1,k_2)\in \Xi_n\}, \qquad \hbox{where}\quad
       z_k = \cos \tfrac{k \pi}{n}.
\end{equation}
Furthermore, we denote by $\Gamma_n^\circ := \Gamma_n \cap (-1,1)^2$ the
subset of interior points of $\Gamma_n$, by $\Gamma_n^e$ the set of points
in $\Gamma_n$ that are on the boundary of $[-1,1]^2$ but not on the four
corners, and by $\Gamma_n^v$ the set of points in $\Gamma_n$ that are on
the corners of $[-1,1]^2$.  The sets $\Xi_n^\circ$, $\Xi_n^e$ and $\Xi_n^v$ are
defined accordingly. A simple counting shows that
\begin{equation} \label{Gamma_n}
  |\Xi_n| = (\lfloor \tfrac{n}2 \rfloor+1)^2 + ( \lfloor \tfrac{n-1}2  \rfloor+1)^2
        = \frac{n(n+1)}{2} +\Big \lfloor \frac{n}{2} \Big\rfloor +1.
\end{equation}

\begin{thm} \label{thm:cubaT}
The cubature formula
\begin{equation} \label{cubaT}
 \frac{1}{\pi^2} \int_{[-1,1]^2} f(t) W_0(t) dt
        = \frac{1}{2 n^2}\sum_{k \in \Xi_n} \l_k^{(n)} f(z_{k_1},z_{k_2}),
        \quad
  \l_k^{(n)} : = \begin{cases} 4, & k \in \Xi_n^\circ, \\
     2, & k \in \Xi_n^e, \\
     1, & k \in \Xi_n^v, \end{cases}
\end{equation}
is exact for $\Pi_{2n-1}^2$.
\end{thm}

\begin{proof}
We note that $X_n^*$ is symmetric in the sense that $k \in X_n^*$
implies that $(-k_1, k_2) \in X_n^*$ and $(k_1,-k_2) \in X_n^*$. Let
$g(x) = f(\cos2\pi x_1,\cos 2\pi x_2)$. Then $g$ is even in each of
its variables and $g(\frac{k}{2n}) = f(z_{k_1}, z_{k_2})$. Notice
that $f \in \Pi_{2n-1}^2$ implies $g \in \CT_{2n-1}^*$. Applying the
first cubature formula \eqref{cuba1} to $g(x)$, we see that
\eqref{cubaT} follows from the following identity,
$$
\sum_{k \in X^*_n}  c_k^{(n)} g(\tfrac{k}{2n}) = \sum_{k \in \Xi_n} \l_k^{(n)}
  f(z_{k_1},z_{k_2}).
$$
To prove this identity, let $k \sigma$ denote the set of distinct
elements in $\{(\pm k_1,\pm k_2)\}$; then $g(\frac{k}{2n})$ takes
the same value on all points in $k \sigma$. If $k \in X_n^*$, $k_1
\ne  0$ and $k_2 \ne 0$, then $k \sigma$ contains 4 points;
$\sum_{j\in k\sigma} c_k^{(n)} g(\frac{j}{2n}) = 4 g(\frac{k}{2n})$
if $k \in X_n^\circ$, $\sum_{j\in k\sigma} c_k^{(n)} g(\frac{j}{2n})
= 2 g(\frac{k}{2n})$ if $k \in X_n^e$, and $\sum_{j\in k\sigma}
c_k^{(n)} g(\frac{j}{2n}) = g(\frac{k}{2n})$ if $k \in X_n^v$. If
$k_1 =0$ and $k_2 \ne 0$ or $k_2 =0$ and $k_1 \ne 0$, then $k
\sigma$ contains 2 points; $\sum_{j\in k\sigma} c_k^{(n)}
g(\frac{j}{2n}) = 2 g(\frac{k}{2n})$ if $k \in X_n^\circ$ and
$\sum_{j\in k\sigma} c_k^{(n)} g(\frac{j}{2n}) =  g(\frac{k}{2n})$
if $k \in X_n^e$. Finally, if $k=(0,0)$ then $k \sigma$ contains 1
point and $g(0,0)$ has coefficient 1. Putting these together proves
the identity.
\end{proof}

By \eqref{Gamma_n},  the number of nodes of the cubature formula \eqref{cubaT}
is just one more than the lower bound \eqref{Moller}.  We can also write
\eqref{cubaT} into a form that is more explicit. Indeed, if $n = 2m$, then
\eqref{cubaT} can be written as
\begin{align} \label{cubaT:even}
 &\frac{1}{\pi^2} \int_{[-1,1]^2} f(t) W_0(t) dt \\
   & \qquad\qquad =   \frac{2}{n^2}
     \sideset{}{''}  \sum_{i=0}^{m} \sideset{}{''}\sum_{j=0}^m f(z_{2i},z_{2j}) +
     \frac{2}{n^2}
       \sum_{i=0}^{m-1} \sum_{j=0}^{m-1} f(z_{2i+1},z_{2j+1}), \notag
\end{align}
where $\sum''$ means that the first and the last terms in the
summation are halved. If $n = 2m+1$, then \eqref{cubaT} can be written as
\begin{align} \label{cubaT:odd}
 &\frac{1}{\pi^2} \int_{[-1,1]^2} f(t) W_0(t) dt \\
   & \qquad\qquad =   \frac{2}{n^2}
     \sideset{}{'}  \sum_{i=0}^{m} \sideset{}{'}\sum_{j=0}^m f(z_{2i},z_{2j}) +
     \frac{2}{n^2}
       \sideset{}{'}  \sum_{i=0}^{m} \sideset{}{'}\sum_{j=0}^{m}
            f(z_{n - 2i},z_{n-2j}), \notag
\end{align}
where $\sum'$ means that the first term in the sum is divided by 2.
The formula \eqref{cubaT:odd} appeared in \cite{X94}, where it was constructed
by considering the common zeros of orthogonal polynomials of two variables.

From the cubature formula \eqref{cuba1}, we can also derive cubature formulas
for the Chebyshev weight $W_1$ of the second kind.

\begin{thm} \label{thm:cubaU}
The cubature formula
\begin{equation} \label{cubaU}
 \frac{1}{\pi^2} \int_{[-1,1]^2} f(t) W_1(t) dt
        = \frac{2}{n^2} \sum_{k \in \Xi_n^\circ}
           \sin^2 \tfrac{k_1 \pi}{n}\sin^2 \tfrac{k_2 \pi}{n} f(z_{k_1},z_{k_2})
\end{equation}
is exact for $\Pi_{2n-5}^2$.
\end{thm}

\begin{proof}
We apply the first cubature formula in \eqref{cuba1} on the functions
$$
   \sin (2 \pi (k_1+1) x_1) \sin( 2\pi (k_2+1) x_2 ) \sin 2 \pi x_1 \sin 2\pi x_2
$$
for $0 \le k_1 + k_2 \le 2n-5$, where $t_1 = \cos 2 \pi x_1$ and
$t_2 = \cos 2 \pi x_2$ as in \eqref{t-x}. Clearly these functions are even in
both $x_1$ and $x_2$ and they are functions in $\CT_{2n-1}^*$. Furthermore,
they are zero when $x_1 =0$ or $x_2 =0$, or when $(x_1,x_2)$ are on the
boundary of $X_n^*$. Hence, the change of variables  \eqref{t-x} shows that
the first cubature in \eqref{cuba1} becomes \eqref{cubaU} for $U_{k_1}(t_1)
U_{k_2}(t_2)$.
\end{proof}

A simple counting shows that $|\Xi_n^\circ| =\lfloor \frac{n}{2}\rfloor^2
+ \lfloor \frac{n-1}{2}\rfloor^2 = \frac{(n-1)(n-2)}{2} + \lfloor \frac{n}{2}\rfloor$. 
The number of nodes of the cubature formula \eqref{cubaU}
is also one more than the lower bound \eqref{Moller}.
In this case, this formula appeared already in \cite{MP}.


\subsection{Interpolation by polynomials}

As shown in \cite{LSX}, there is a close relation between interpolation and
discrete Fourier transform. We start with a simple result on interpolation by
trigonometric functions in $\CT_n$.

\begin{prop} \label{prop:3.4}
For $n \ge 1$ define
\begin{equation}\label{CI_n}
    I_n f(x) : = \sum_{k \in X_n} f(\tfrac{k}{2n}) \Phi_n(x- \tfrac{k}{2n}),
        \qquad
         \Phi_n(x) := \frac{1}{2n^2} \sum_{\nu \in \Lambda_n} \e_\nu(x).
\end{equation}
Then $I_n f(\tfrac{k}{2n}) = f(\tfrac{k}{2n})$ for all $k \in X_n$.
\end{prop}

\begin{proof}
For $j \in \Lambda_n$ define $k = 2n B^{-\tr} j$. From the relation \eqref{j-k},
$j \in \Lambda_n$ is equivalent to $k \in X_n$ with $k=2n B^{-\tr}j$.
As a result, we can write $I_n f(x)$ as
$$
   I_n f(x) =  \sum_{j \in \Lambda_n} f(B^{-\tr} j ) \Phi_n (x- B^{-\tr} j)
$$
and the interpolation means $I_n f(B^{-\tr}j ) = f(B^{-\tr} j)$ for $j \in \Lambda_n$.
For $k,j \in \Lambda_n$,
\begin{align*}
   \Phi_n(B^{-\tr}(j -k)) & = \frac{1}{2n^2}\sum_{\nu \in \Lambda_n} \e_\nu(B^{-\tr}(j-k))
   = \delta_{k,j}
\end{align*}
by \eqref{d-ortho2}.
\end{proof}

For our main result, we need a lemma on the symmetric set $X_n^*$ and
$\Lambda_n^*$. Recall that $c_k^{(n)}$ is defined for $k \in X_n^*$. Since
the relation \eqref{j-k} clearly extends to
\begin{equation} \label{j-k*}
j\in \Lambda_n^*  \qquad \Longleftrightarrow \qquad k = 2n B^{-\tr} j \in X_n^*,
\end{equation}
we define $\widetilde c_j^{(n)} = c_k^{(n)}$ whenever $k$ and $j$ are so
related. Comparing to \eqref{CI_n}, we then define
\begin{equation}\label{CI_n*}
    I_n^* f(x) : = \sum_{k \in X_n^*} f(\tfrac{k}{2n}) \Phi_n^*(x- \tfrac{k}{2n}),
        \quad \hbox{where} \quad
   \Phi_n^*(x) := \frac{1}{2n^2} \sum_{\nu \in \Lambda_n^*}
                                 \wt c_\nu^{(n)} \e_\nu(x).
\end{equation}
We also introduce the following notation: for $k \in X_n^e$, we denote
by $k'$ the point on the opposite edge of $X_n^*$; that is, $k' \in
X_n^e$ and $k'  = k \pm (2n,0)$ or $k' = k \pm (0,2n)$. Furthermore, we
denote by $j'$ the index corresponding to $k'$ under \eqref{j-k*}.

\begin{lem} \label{prop:3.5}
The function $I_n^* f \in \CT_n^*$ satisfies
$$
I_n^* f(\tfrac{k}{2n}) = \begin{cases} f(\tfrac{k}{2n}), & k \in X_n^\circ, \\
    f(\tfrac{k}{2n}) +  f(\tfrac{k'}{2n}), & k \in X_n^e, \\
     f(\tfrac{k}{2n})+ f(\tfrac{(-k_1,k_2)}{2n})+f(\tfrac{(k_1,-k_2)}{2n})+f( \tfrac{-k}{2n}),
      & k \in X_n^v.
\end{cases}
$$
\end{lem}

\begin{proof}
As in the proof of the previous theorem, we can write $I^*_n f$ as
$$
   I_n^* f(x) =  \sum_{j \in \Lambda_n^*} f(B^{-\tr} j ) \Phi_n^* (x- B^{-\tr} j)
$$
by using \eqref{j-k*}. Let $S_k(x) = \Phi^*_n(B^{-\tr}j)$.
For all $k, j \in \Lambda^*_n$,
$$
   S_k(B^{-\tr} j ) = \frac{1}{2n^2} \sum_{\nu \in \Lambda_n^*}
        \wt c_\nu^{(n)} \e_{\nu}(B^{-\tr}(j-k)).
$$
Since $\e_{\nu}(B^{-\tr}j) = \e_{\mu}(B^{-\tr}j) $ for any $\mu\equiv \nu \mod B$,
we derive by using a similar argument as in Theorem \ref{inner2} that
$$
   S_k(B^{-\tr} j ) = \frac{1}{2n^2} \sum_{\nu \in \Lambda_n}
         \e_{\nu}(B^{-\tr}(j-k)).
$$
By \eqref{d-ortho2},
$S_k(B^{-\tr} j)  = \delta_{k,j}$
if $k, j \in \Lambda_n$. If $j \in \Lambda_n^* \setminus \Lambda_n$
then $j' \in \Lambda_n$, so that if $k \in \Lambda_n$ then $S_k(B^{-tr}j)
= \delta_{k, j'}$. The same holds for the case of
$j\in \Lambda_n$ and $k \in  \Lambda_n^* \setminus \Lambda_n$.
 If both $k, j \in \Lambda_n^* \setminus \Lambda_n$,
then  $S_k(B^{-\tr} j) =  \delta_{k', j'}$. Using the
relation \eqref{j-k*}, we have shown that $\Phi_n^*(\frac{j-k}{2n}) = 1$
when $k \equiv j \mod 2 n \ZZ^2$ and 0 otherwise, from which the stated
result follows.
\end{proof}

It turns out that the function $\Phi_n^*$ satisfies a compact formula.
 Let us define an operator
$\CP$ by
$$
(\CP f) (x) =\frac{1}{4} \left[ f(x_1,x_2) + f(-x_1,x_2) + f(x_1, -x_2) + f(-x_1,-x_2)\right].
$$
For $\e_k(x) = \e^{2 \pi i k\cdot x}$, it follows immediately that
\begin{equation} \label{Pe_k}
(\CP \e_k)(x) =  \cos (2\pi k_1 x_1)\cos (2\pi k_2 x_2)
    \qquad  \hbox{forall $k \in \ZZ^2$}.
\end{equation}

\begin{lem} \label{lem:D_n}
For $n \ge 0$,
\begin{equation} \label{Phi*formula}
 \Phi_n^*(x) = 2 \left[ D_n(x) + D_{n-1}(x) \right] -
    \frac{1}{4} (\cos 2\pi n x_1 + \cos 2\pi n x_2 ),
\end{equation}
where
\begin{equation}\label{D_n}
   D_n(x) :=  \frac14 \sum_{\nu \in \Lambda_n^*} \e_\nu(x)
               =   \frac12 \frac{\cos \pi (2n+1) x_1\cos \pi x_1-
         \cos \pi (2n+1)x_2\cos \pi x_2} {\cos2\pi x_1 - \cos2\pi x_2}.
\end{equation}
\end{lem}

\begin{proof}
Using the values of $\wt c_\nu^{(n)}$ and the definition of $D_n$, it is easy
to see that
$$
   \Phi_n^*(x) =  2 \left[ D_n(x) + D_{n-1}(x) \right] -
               \sum_{\nu \in \Lambda^v} \e_v(x).
$$
Since $\Lambda_n^v$ contains four terms, $(\pm n, 0)$ and $(0,\pm n)$,
the sum over $\Lambda_n^v$ becomes the second term in \eqref{Phi*formula}.
On the other hand, using the symmetry of $\Lambda_n^*$ and \eqref{Pe_k},
$$
  D_n (x) =  \frac14 \sum_{\nu \in \Lambda_n^*} (\CP \e_\nu)(x) =
              \sideset{}{'} \sum_{0 \le j_1+j_2 \le n} \cos 2\pi j_1 x_1 \cos 2\pi j_2 x_2,
$$
where $\sum'$ means that the terms in the sum are halved whenever either
$j_1 = 0$ or $j_2 =0$, from which the second equal sign in \eqref{D_n}
follows from \cite[(4.2.1) and (4.2.7)]{X95}.
\end{proof}

Our main result in this section is interpolation over points in $\{\frac{k}{2n}:
k \in \Xi_n\}$ with $\Xi_n$ defined in \eqref{Xi}.

\begin{thm}
For $n \ge 0$ define
$$
  \CL_n f(x) = \sum_{k \in \Xi_n} f(\tfrac{k}{2n}) \ell_k(x), \qquad
      \ell_k(x): = \lambda_k^{(n)} \CP \left[\Phi_n^*(\cdot - \tfrac{k}{2n})\right](x)
$$
with $\l_k^{(n)}$ given in \eqref{cubaT}. Then $\CL_n f \in \CT_n$ is even in both variables and it satisfies
$$
  \CL_n f(\tfrac{j}{2n}) = f(\tfrac{j}{2n})  \qquad \hbox{for all} \quad j \in \Xi_n.
$$
\end{thm}

\begin{proof}
As shown in the proof of Proposition \ref{prop:3.5}, $R_k(x):=
\Phi_n^*(x- \tfrac{k}{2n})$ satisfies $R_k(\tfrac{j}{2n}) = 1$  when
$k \equiv j \mod 2n \ZZ^2$ and 0 otherwise. Hence, if $j \in \Xi_n^\circ$
 then
$(\CP R_k)(\tfrac{j}{2n}) = \frac{1}{4} R_k(\tfrac{j}{2n}) = [\l_k^{(n)}]^{-1} \d_{k,j}$.
If $j \in \Xi_n^e$ then the number of terms in the sum of $(\CP R_k)(\tfrac{j}{2n})$
depends on whether $j_1 j_2$ is zero; if $j_1j_2 \ne 0$ then
$(\CP R_k)(\tfrac{j}{2n}) = \frac{1}{4} \left[R_k(\tfrac{j}{2n})+R_k(\tfrac{j'}{2n})\right]
 = \frac12 \d_{k,j}= [\l_k^{(n)}]^{-1} \d_{k,j}$, whereas if $j_1j_2 =0$ then
$(\CP R_k)(\tfrac{j}{2n}) = \frac{1}{2} R_k(\tfrac{j}{2n}) =[\l_k^{(n)}]^{-1} \d_{k,j}$.
For $j = (n,0)$ or $(0,n)$ in $\Xi_n^v$, we have
$(\CP R_k)(\tfrac{j}{2n})= \frac{1}{2} \left[R_k(\tfrac{j}{2n})+R_k(\tfrac{j'}{2n})\right]
 = \d_{k,j}$; for $j = (n,n) \in \Xi_n^v$ we have
$(\CP R_k)(\tfrac{j}{2n})= \frac{1}{4} \left[R_k(\tfrac{(n,n)}{2n})+R_k(\tfrac{(-n,n)}{2n})+R_k(\tfrac{(n,-n)}{2n})
+R_k(\tfrac{(-n,-n)}{2n})\right]
 = \d_{k,j}$;  finally for $j =0 \in \Xi_n^v$, it is evident that $(\CP R_k)(0) =
 \delta_{k,0}$. Putting these together, we have verified that $\ell_k(\frac{j}{2n})
 = \d_{k,j}$ for all $j,k \in \Xi_n^*$, which verifies the interpolation of $\CL_n f$.
\end{proof}

As in the case of cubature, we can translate the above theorem to interpolation
by algebraic polynomials by applying the change of variables \eqref{t-x}. Recall
$\Gamma_n$ defined in \eqref{G_n}.

\begin{thm}
For $n \ge 0$, let
$$
   \CL_n f (t) = \sum_{z_k \in \Gamma_n} f(z_k) \ell^*_k(t),
         \qquad
         \ell_k^*(t) = \ell_k(x) \quad\hbox{with}\quad  t_i = \cos 2 \pi x_i,\ i=1,2.
$$
Then $\CL_n f \in \Pi_n^2$ and it satisfies $\CL_n f(z_k) = f(z_k)$ for all
$z_k \in \Gamma_n$. Furthermore, under the change of variables \eqref{t-x},
the fundamental polynomial $\ell_k^*(t)$ satisfies
$$
 \ell_k^*(t) = \frac12\CP \left[D_n(\cdot -\tfrac{k}{2n})
                +D_{n-1}(\cdot -\tfrac{k}{2n}) \right](x)
     - \frac{1}{4} \left[(-1)^{k_1} T_{k_1}(t_1)+(-1)^{k_2} T_{k_2}(t_2) \right].
$$
\end{thm}

\begin{proof}
That $\CL_n f$ interpolates at $z_k \in \Gamma_n$ is an immediate consequence
of the change of variables, which also shows that $\CL_n f \in \Pi_n^2$. Moreover,
$\cos 2\pi n(x_1 - \frac{k_1}{2n})  = (-1)^{k_1} \cos 2\pi n x_1 =
(-1)^{k_1}T_n(x_1)$, which verifies the formula of $\ell_k^*(t)$.
\end{proof}

The polynomial $\CL_n f$ belongs, in fact, to a subspace $\Pi_n^* \subset
\Pi_n^2$ of dimension $|\Xi_n| = \dim \Pi_{n-1}^2 + \lfloor \frac{n}2 \rfloor +1$,
and it is the unique interpolation polynomial in $\Pi_n^*$. In the case of
$n$ is odd, this interpolation polynomial was defined and studied in
\cite{X96}, where a slightly different scheme with one point less was studied
in the case of even $n$. Recently the interpolation polynomials in \cite{X96}
have been tested and studied numerically in \cite{Xpt1,Xpt2}; the results
show that these polynomials can be evaluated efficiently and provide valuable
tools for numerical computation.


\section{Cubature and Interpolation on the cube}
\setcounter{equation}{0}

For $d =2$, the choice of our spectral set $\Omega_B$ and lattice in the
previous section ensures that we end up with a space close to the polynomial
subspace $\Pi_n^2$; indeed, monomials in $\Pi_n^2$ are indexed by
$0 \le j_1 + j_2 \le n$, a quarter of $\Lambda_n^*$.  For $d = 3$, the
same consideration indicates that we should choose the spectral set as the
octahedron $\{x: -n \le x_1 \pm x_2 \pm x_3 \le n\}$. The octahedron, however,
does not tile $\RR^3$ by lattice translation (see, for example, \cite[p.  452]{CS}).
As an alternative, we choose the spectral set as {\it rhombic dodecahedron},
which tiles $\RR^3$ by lattice translation with {\it face centered cubic (fcc) lattice}.
In \cite{LX}, a discrete Fourier analysis on the rhombic dodecahedron is
developed and used to study cubature and interpolation on the rhombic
dodecahedron, which also leads to results on tetrahedron. In contrast, our results
will be established on the cube $[-\frac12, \frac12]^3$, but our set $\Omega_B$
is chosen to be a rhombic dodecahedron.

\subsection{Discrete Fourier analysis and cubature formula on the cube}
We choose our matrix $B$ as the generator matrix of fcc lattice,
\begin{align*}
   B = n \begin{pmatrix}
               0  &  1 &  1 \\
               1  &  0 &  1 \\
               1  &  1 &  0
              \end{pmatrix}
     \quad \hbox{and} \quad
    B^{-1} = \frac{1}{2n} \begin{pmatrix}
               -1 &  1  &  1 \\
               1  &  -1 &  1 \\
               1  &  1  &  -1
              \end{pmatrix}.
\end{align*}
The spectral set of the fcc lattice is the rhombic dodecahedron
(see Figure 2). Thus,
$$
\Omega_B = \{x \in \RR^3:  -n \le x_\nu \pm x_\mu < n,  1 \le \nu<\mu \le 3\}.
$$

\begin{figure}[ht]
\centering
\includegraphics[width=0.6\textwidth]{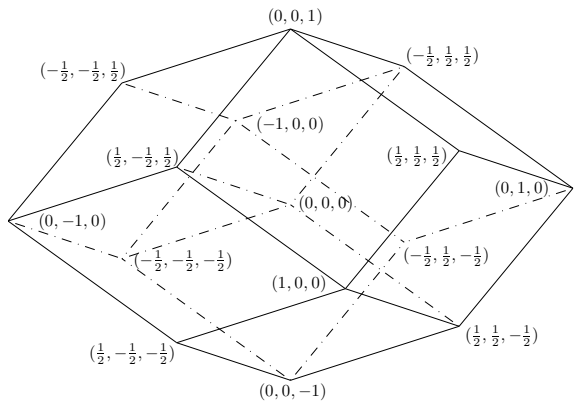}
\caption{Rhombic dodecahedron}
\end{figure}

The strict inequality in the definition of $\Omega_B$ reflects our requirement
that the tiling of the spectral set has no overlapping. From the expression of
$B^{-\tr}$, it follows that $\Lambda_B =: \Lambda_n $ is given by
$$
  \Lambda_n := \{j \in \ZZ^3: -n \le -j_1 + j_2 +j_3,
    j_1 - j_2 +j_3, j_1 + j_2 -j_3 < n\}.
$$
It is known that $|\Lambda_n| = \det (B) = 2 n^3$. Furthermore, $\Lambda_B^\dag
 = : \Lambda_n^\dag$ is given by
$$
   \Lambda_n^\dag = \ZZ^3 \cap \Omega_B =
       \{k \in \ZZ^3:  -n \le k_\nu \pm k_\mu < n,  1 \le \nu<\mu \le 3 \}.
$$
We denote the space $\CT_B$ by $\CT_n$, which is given by
$$
   \CT_n: = \ \mathrm{span}
      \left\{\e^{2\pi i\, k\cdot x} : \ k \in \Lambda_{n}^\dag \right\}.
$$
Then $\dim \CT_n = |\Lambda_n^\dag| = \det(B)=2 n^3$.

\begin{thm} \label{inner1-d3}
Define the set
\begin{align*}
X_n: =  \left \{  2 k: - \tfrac{n}{2} \le k_1, k_2, k_3 < \tfrac{n}{2}\right\}
     \cup  \left\{ 2 k+1: - \tfrac{n+1}{2}  \le k_1, k_2,k_3 < \tfrac{n-1}{2} \right\}.
\end{align*}
Then for all $f,g \in \CT_n$,
$$
  \langle f, g \rangle_n : = \frac{1}{2n^3} \sum_{k\in X_n} f(\tfrac{k}{2n})
         \overline{g(\tfrac{k}{2n})}
     = \int_{[-\frac12,\frac12]^3} f(x)\overline{g(x)} dx.
$$
\end{thm}

\begin{proof}
Changing variables from $j$ to $k = 2n B^{-\tr}j$, or $j = B^\tr k/(2n)$,
then, as $j_1, j_2, j_3$ are integers and $j_1 = \frac{k_2+k_3}{2}$,
$j_2 = \frac{k_1+k_3}{2}$, $j_3 = \frac{k_1+k_2}{2}$, we see that
\begin{equation} \label{j-k3}
     j\in \Lambda_n  \quad \Longleftrightarrow  \quad 2n B^{-\tr}j \in X_n
     \quad \hbox{and} \quad  \sum_{j \in \Lambda_n} f(B^{-tr} j)
= \sum_{k\in X_n} f(\tfrac{k}{2n}),
\end{equation}
from which we conclude that $\la f,g\ra_n = \la f, g \ra_B$.
Consequently, this theorem is a special case of Theorem \ref{thm:2.1}.
\end{proof}

Just like the case of $d =2$, we denote the symmetric counterpart of
$X_n$ by $X_n^*$ which is defined by
$$
X_n^* : = \left \{ 2 k: - \tfrac{n}{2} \le k_1, k_2,k_3 \le \tfrac{n}{2}\}   \cup
    \{ 2 k+1:  - \tfrac{n+1}{2}  \le k_1, k_2, k_3\le \tfrac{n-1}{2}\right\}\!.
$$
A simple counting shows that $|X_n^*| = n^3 + (n+1)^3$.
The set $X_n^*$ is further partitioned into four parts,
$$
   X_n^* = X_n^\circ  \cup X_n^f \cup X_n^e \cup X_n^v,
$$
where $X_n^{\circ} = X_n^* \cap (-n,n)^2$ is the set of interior points,
$X_n^f$ contains the points in $X_n^*$ that are on the faces of $[-n,n]^3$
but not on the edges or vertices, $X_n^{e}$ contains the points in $X_n^*$
that are on the edges of $[-n,n]^3$ but not on the corners or vertices,
while $X_n^v$ denotes the points of $X_n^*$ at the vertices of  $[-n,n]^3$.

\begin{thm} \label{inner2-3d}
Define the inner product
\begin{equation} \label{eq:inner2-3d}
\la f, g \ra_n^* : = \frac{1}{2 n^3} \sum_{k \in X_n^*} c_k^{(n)} f(\tfrac{k}{2n})
          \overline{g(\tfrac{k}{2n})},
    \qquad \hbox{where}\quad
          c_k^{(n)} = \begin{cases} 1, & k \in X_n^\circ \\ \frac12, & k \in X_n^f \\
                      \frac14, & k \in X_n^e \\ \frac18, & k \in X_n^v   \end{cases}.
\end{equation}
Then for all $f,g \in \CT_n$,
$$
\int_{[-\frac12,\frac12]^3} f(x)\overline{g(x)} dx = \langle f, g \rangle_n
        = \langle f, g \rangle_n^*.
$$
\end{thm}

\begin{proof}
The proof follows along the same line as the proof of Theorem \ref{inner2}.
We only need to show $\la f,g\ra_n = \la f, g\ra_n^*$ if $f\overline{g}$ is periodic.
The interior points of $X_n$ and $X_n^*$ are the same, so that $c_k^{(n)} = 1$
for $k \in X_n^\circ$. Let $\ve_1 = (1,0,0)$, $\ve_2=(0,1,0)$, and $\ve_3=(0,0,1)$.
Each point $k$ in $X_n^f$ has exactly one opposite point $k^*$ in $X_n^f$ under translation by $\pm n \ve_i$ and only one of them is in $X_n$, so that
$f(x_k) = \frac{1}{2} [f(x_k)+f(x_k^*)]$ if $f$ is periodic, which is why we define
$c_k^{(n)} =\frac12$ for $k \in X_n^f$. Evidently, only three edges of $X_n^*$ are
in $X_n^* \setminus X_n$. Each point in $X_n^e$ corresponds to exactly four
points in $X_n^e$ under integer translations $\pm n \ve_i$ and only one among
the four is in $X_n$, so we define $c_k^{(n)} = \frac14$ for $k \in X_n^e$.
Finally, all eight corner points can be derived from translations $n \ve_i$ points,
used repeatedly, and exactly one, $(-n,-n,-n)$, is in $X_n^*\setminus X_n$, so that
we define $c_k^{(n)} = \frac18$ for $k \in X_n^v$.
\end{proof}

We also denote the symmetric counterpart of $\Lambda_n^\dag$ by
$\Lambda_n^{\dag*}$,
\begin{equation}\label{LambdaDag*}
  \Lambda_n^{\dag*} :=\{j \in \ZZ^3:  -n \le j_\nu \pm j_\mu \le n,  1 \le \nu<\mu \le 3 \}
\end{equation}
and denote the counterpart of $\CT_n$ by $\CT_n^*$, which is defined accordingly by
$$
  \CT_n^* : = \mathrm{span}
           \left\{\e^{2\pi i\, k\cdot x} : \ k \in \Lambda_n^{\dag*} \right\}.
$$

\begin{thm} \label{thm:cuba1-3d}
For $n \ge 2$, the cubature formulas
\begin{equation} \label{cuba1-3d}
     \int_{[-\frac12,\frac12]^3} f(x) dx =
           \frac{1}{2n^3} \sum_{k \in X_n^*} c_k^{(n)} f(\tfrac{k}{2n})
 \quad\hbox{and} \quad
     \int_{[-\frac12,\frac12]^3} f(x) dx =  \frac{1}{2n^3} \sum_{k \in X_n} f(\tfrac{k}{2n})
\end{equation}
are exact for $f \in \CT_{2n-1}^*$.
\end{thm}

\begin{proof}
As in the proof of Theorem \ref{thm:cuba1},  for any $j\in \ZZ^3$, there exist  
$\nu\in \Lambda_n^{\dag}$ and $l\in \ZZ^3$ such that $j=\nu+Bl$.

Assume now $j\in \Lambda^{\dag *}_{2n-1}$. Clearly  the integral of $\e_j$ 
over $\Omega$ is $\delta_{j,0}$. On the other hand, let us suppose $j=\nu+Bl$ 
with $\nu\in \Lambda_n$ and $l\in \ZZ^3$. Then it is easy to see that
$\e_j(\frac{k}{2n}) = \e_{\nu}(\frac{k}{2n})$ for each $k\in X^*_n$.
Consequently, we get from Theorem \ref{inner2-3d} that
\begin{align*}
   \sum_{k\in X_n^*} c_{k}^{(n)}\e_{j}(\tfrac{k}{2n}) & = \sum_{k\in X_n^*}
     c_{k}^{(n)}\e_{\nu}(\tfrac{k}{2n})  = \sum_{k\in X_n}
     \e_{\nu}(\tfrac{k}{2n}) \\ 
     & = \sum_{k\in X_n} \e_{j}(\tfrac{k}{2n})  
        = \int_\Omega  \e_{\nu}(x) dx = \delta_{\nu,0}.
\end{align*}
Since $\nu=0$ implies  $j=l=0$, we further  obtain that $\delta_{\nu,0}=
\delta_{j,0}$. This states that the cubature formulas \eqref{cuba1-3d} are 
exact for each $\e_j$ with $j\in \Lambda^{\dag *}_{2n-1}$, which completes
the proof.
\end{proof}

\subsection{Cubature formula for algebraic polynomials}
We can also translate the cubature in Theorem \ref{thm:cuba1-3d} into one
for algebraic polynomials. For this we use the change of variables
\begin{equation} \label{t-x3d}
   t_1 = \cos 2 \pi x_1,  \quad t_2 = \cos 2 \pi x_2,  \quad t_3 = \cos 2 \pi x_3, \quad
       x \in [-\tfrac12,\tfrac12]^3.
\end{equation}
Under \eqref{t-x3d}, the functions $\cos2\pi k_1 x_1 \cos 2\pi k_2 x_2
\cos 2\pi k_3x_3$ become algebraic polynomials $T_{k_1}(t_1)T_{k_2}(t_2)
T_{k_3}(t_3)$, which are even in each of its variables. The subspace of
$\CT_n^*$ that consists of functions that are even in each of its variables
corresponds to the polynomial subspace
$$
   \Pi_n^*: = \mathrm{span} \{T_{k_1}(x_1)T_{k_2}(x_2)T_{k_3}(x_3):
         k_1, k_2,k_3 \ge 0,  \, k_\nu+k_\mu \le n, \, 1 \le \nu < \mu \le n\}.
$$
Notice that $X_n^*$ is symmetric in the sense that if $x\in X_n^*$ then $\sigma x
\in X_n^*$ for all $\sigma \in \{-1,1\}^3$, where $(\sigma x)_i = \sigma_i x_i$. In
order to evaluate functions that are even in each of its variables on $X_n^*$
we only need to consider $X_n^* \cap \{x: x_1,x_2,x_3 \ge 0\}$.  Hence, we
define,
\begin{equation} \label{Xi_n}
   \Xi_n : = \{2 k: 0 \le k_1,k_2,k_3 \le \tfrac{n}{2}\} \cup
        \{2 k+1: 0 \le k_1,k_2,k_3 \le \tfrac{n-1}{2}\}
\end{equation}
and, under the change of variables \eqref{t-x3d}, define
\begin{equation} \label{Gamma_n3d}
    \Gamma_n :=\{(z_{k_1},z_{k_2},z_{k_3}): k \in \Xi_n \}, \qquad z_k = \tfrac{k}{2n}.
\end{equation}
Moreover, we denote by $\Gamma_n^\circ$, $\Gamma_n^f$, $\Gamma_n^e$ and
$\Gamma_n^v$ the subsets of $\Gamma_n$ that contains interior points, points
on the faces but not on the edges, points on the edges but not on the vertices,
and points on the vertices, of $[-1,1]^3$, respectively, and we define
$\Xi_n^\circ$, $\Xi_n^f$, $\Xi_n^e$ and $\Xi_n^v$ accordingly. A simple counting
shows that
\begin{equation} \label{Xi3d}
    |\Xi_n| =  (\lfloor \tfrac{n}2 \rfloor+1)^3 + ( \lfloor \tfrac{n-1}2  \rfloor+1)^3
        = \begin{cases} \frac{(n+1)^3}{4} + \frac{3(n+1)}{4},  & \hbox{$n$ is even},\\
              \frac{(n+1)^3}{4},  & \hbox{$n$ is odd}. \end{cases}
\end{equation}

\begin{thm} \label{thm:cubaT3d}
Write $z_k =(z_{k_1}, z_{k_2},z_{k_3})$. The cubature formula
\begin{equation} \label{cubaT3d}
 \frac{1}{\pi^3} \int_{[-1,1]^3} f(t) W_0(t) dt
        = \frac{1}{2 n^2}\sum_{k \in \Xi_n} \l_k^{(n)} f(z_k),
        \quad
  \l_k^{(n)} : = \begin{cases} 8, & k \in \Xi_n^\circ, \\
     4, & k \in \Xi_n^f, \\
     2, & k \in \Xi_n^e, \\
     1, & k \in \Xi_n^v, \end{cases}
\end{equation}
is exact for $\Pi_{2n-1}^*$. In particular, it is exact for $\Pi_{2n-1}^3$.
\end{thm}

\begin{proof}
Let $g(x) = f(\cos2\pi x_1,\cos 2\pi x_2, \cos 2\pi x_3)$. Then $g$ is even in each of its
variables and $g(\frac{k}{2n}) = f(z_k)$. Applying the first cubature formula
in \eqref{cuba1} to $g(x)$, we see that \eqref{cubaT} follows from the
following identity,
$$
\sum_{k \in X^*_n}  c_k^{(n)} g(\tfrac{k}{2n}) = \sum_{k \in \Xi_n} \l_k^{(n)}
  f(z_k).
$$
This identity is proved in the same way that the corresponding identity in
Theorem \ref{thm:cubaT} is proved. Let $k \sigma$ denote the set of distinct
elements in $\{ k \sigma : \sigma \in \{-1,1\}^3\}$; then $g(\frac{k}{2n})$ takes the
same value on all points in $k \sigma$. If $k \in X_n^*$, $k_i \ne  0$ for $i =1,2,3$,
then $k \sigma$ contains 8 points; if exactly one $k_i$ is zero then $k \sigma$
contains 4 points; if exactly two $k_i$ are zero then $k \sigma$ contains one point;
and, finally, if $k=(0,0,0)$ then $k \sigma$ contains one point. In the case of
$k_i \ne 0$ for $i=1,2,3$, $\sum_{j\in k\sigma} c_k^{(n)}
g(\frac{j}{2n}) = 8 g(\frac{k}{2n})$ if $k \in X_n^\circ$,
$\sum_{j\in k\sigma} c_k^{(n)} g(\frac{j}{2n}) = 4 g(\frac{k}{2n})$ if $k \in X_n^f$,
and $\sum_{j\in k\sigma} c_k^{(n)} g(\frac{j}{2n}) = 2 g(\frac{k}{2n})$ if $k \in X_n^e$.
The other cases are treated similarly. Thus, \eqref{cubaT3d} holds for
$\Pi_{2n-1}^*$.

Finally, the definition of $\Pi_n^*$ shows readily that it
contains
$$
   \Pi_n^3 = \mathrm{span} \{ T_{k_1}(x_1)T_{k_2}(x_2)T_{k_3}(x_3):
       k_1,k_2,k_3 \ge 0, \, 0 \le k_1+k_2+k_3 \le n\}
$$
as a subspace. In particular, $\Pi_{2n-1}^*$ contains $\Pi_{2n-1}^3$ as a
subset.
\end{proof}

We note that $\Pi_{2n-1}^*$ contains $\Pi_{2n-1}^3$ as a subspace, but it does
not contain $\Pi_{2n}^3$ since $T_{n}(x_1)T_n(x_2)$ is in $\Pi_{2n}^3$ but
not in $\Pi_{2n-1}^3$. Hence, the cubature \eqref{cubaT3d} is of degree $2n-1$.
A trivial cubature formula of degree $2n-1$ for $W_0$ can be derived by
taking the product of Gaussian quadrature of degree $2n-1$ in one variable,
which has exactly $n^3$ nodes. In contrast, according to \eqref{Xi3d},  the number
of nodes of our cubature \eqref{cubaT} is in the order of $n^3/4 + \CO(n^2)$,
about a quarter of the product formula. As far as we know, this is the best
that is available at the present time. On the other hand, the lower bound for
the number of nodes states that a cubature formula of degree $2n-1$ needs
at least $n^3/6 + \CO(n^2)$ nodes. It is, however, an open question if
there exist formulas with number of nodes attaining this theoretic lower bound.

Recall the cubature \eqref{cubaT3d} is derived by choosing the spectral set
as a rhombic dodecahedron. One natural question is how to choose a spectral
set that tiles $\RR^3$ by translation so that the resulted cubature formula is
of degree $2n-1$ and has  the smallest number of nodes possible. Among the
regular lattice tiling, the rhombic dodecahedron appears to lead to the
smallest number of nodes.

Just as Theorem \ref{thm:cubaU}, we can also derive a cubature formula of
degree $2n-5$ for $W_1$ from Theorem \ref{thm:cuba1-3d}. We omit the proof
as it follows exactly as in Theorem \ref{thm:cubaU}.

\begin{thm} \label{thm:cubaU3d}
The cubature formula
\begin{equation} \label{cubaU3d}
 \frac{1}{\pi^3} \int_{[-1,1]^3} f(t) W_1(t) dt
        = \frac{4}{n^3} \sum_{k \in \Xi_n^\circ}
        \sin^2 \tfrac{k_1 \pi}{n}\sin^2 \tfrac{k_2 \pi}{n} \sin^2 \tfrac{k_3 \pi}{n}f(z_k)
\end{equation}
is exact for $\Pi_{2n-5}^*$; in particular, it is exact for $\Pi_{2n-5}^3$.
\end{thm}

\subsection{A compact formula for a partial sum}
In order to obtain the compact formula for the interpolation function, we
follow \cite{LX} and use homogeneous coordinates and embed the rhombic
dodecahedron into the plane $t_1+t_2+t_3+t_4 =0$ of $\RR^4$. Throughout
the rest of this paper, we adopt the convention of using bold letters,
such as $\tb$, to denote the points in the space
\begin{align*}
 \RR_H^4 := \left\{ \tb = (t_1,t_2,t_3,t_4)\in \RR^4:
  t_1+t_2+t_3+t_4=0 \right\}.
\end{align*}
In other words, the bold letters such as $\tb$ and $\kb$ will always mean
homogeneous coordinates.  The transformation between $x \in \RR^3$ and
$\tb \in \RR_H^4$ is defined by
\begin{equation}\label{coordinate1}
      \begin{cases} x_1= t_2 + t_3 \\
            x_2 =   t_1+ t_3 \\
            x_3=    t_2 + t_1
          \end{cases}
         \quad \Longleftrightarrow \quad
       \begin{cases} t_1=  \tfrac{1}{2}(-x_1+x_2+x_3) \\
            t_2 =   \tfrac{1}{2}(x_1-x_2+x_3) \\
            t_3=   \tfrac{1}{2} (x_1+x_2-x_3) \\
            t_4 =  \tfrac{1}{2}(-x_1-x_2-x_3).
       \end{cases}
 \end{equation}
In this homogenous coordinates, the spectral set  $\Omega_B$ becomes
\begin{equation}\label{Omega}
   \Omega_B = \left\{\tb \in \RR_H^4:
             -1< t_i-t_j  \leq  1,  1\leq i < j \leq 4\right\}.
\end{equation}

We now use homogeneous coordinates to describe $\Lambda_n^{\dag*}$
defined in \eqref{LambdaDag*}. Let $\ZZ_H^4 : = \ZZ^4 \cap \RR_H^4$ and
$$
   \HH:= \{\jb \in \ZZ_H^4:   j_1 \equiv j_2 \equiv j_3 \equiv j_4 \mod 4\}.
$$
In order to keep the elements as integers, we make the change of variables
\begin{align} \label{eq:j-k}
\begin{split}
 & j_1= 2 (-k_1+k_2+k_3), \, \quad j_2= 2 (k_1-k_2+k_3),\\
 & j_3= 2 (k_1+k_2-k_3),  \qquad j_4= 2 (-k_1-k_2-k_3)
\end{split}
\end{align}
for $k=(k_1,k_2,k_3)\in \Lambda_n^{\dag*}$. It then follows that
$\Lambda_n^{\dag*}$ in homogeneous coordinates becomes
\begin{equation*}
 \GG_n:= \{j\in \HH:  j_1 \equiv j_2 \equiv j_3 \equiv j_4 \equiv 0 \mod 2,\,
     - 4 n \le j_\nu- j_\mu \le 4 n, 1 \le \nu,\mu \le 4 \}.
\end{equation*}
We could have changed variables without the factor 2, setting
$j_1 = -k_1+k_2+k_3$ etc. We choose the current change of variables so
that we can use some of the computations in \cite{LX}. In fact, the set
\begin{equation}\label{HHn}
  \HH_n^* := \{j\in \HH:  - 4 n \le j_\nu- j_\mu \le 4 n, 1 \le \nu,\mu \le 4 \}
\end{equation}
is used in \cite{LX}. The main result of this subsection is a compact
formula for the partial sum
\begin{equation}\label{Dn3D}
  D_n(x) := \sum_{k \in \Lambda_n^{\dag*}} \e_k(x)
     = \sum_{\jb \in \GG_n} \e_\jb(\tb) =: D^*_n(\tb),
    \qquad \e_\jb(\tb) := \e^{\frac{\pi i}{2} \jb \cdot \tb},
\end{equation}
where $x$ and $\tb$ are related by \eqref{coordinate1} and the middle
equality follows from the fact that $\Lambda_n^{\dag*} = \GG_n$ under
this change of variables. In fact, by \eqref{coordinate1} and
\eqref{eq:j-k}, we have
\begin{align*}
k \cdot x &=  k_1 (t_2+t_3)+ k_2 (t_1+t_3)+ k_3 (t_1+t_2) \\
       &=  (k_2+k_3)t_1+  (k_1+k_3)t_2+  (k_1+k_2) t_3  \\
      & =  \frac{1}{4}\left[(j_1-j_4) t_1+  (j_2-j_4)t_2+  (j_3-j_4) t_3 \right]
      =  \frac{1}{4} \jb \cdot \tb
\end{align*}
where in the last step we have used the fact that $\tb \in \RR_H^4$.
The compact formula of $D_n(\tb)$ is an essential part of the compact
formula for the interpolation function.

\begin{thm}
For $n \ge 1$,
$$
   D_n^*(\tb) = \Theta_{n+1}(\tb) - \Theta_{n}(\tb) - \left(  \Theta_n^{\mathrm{odd}}(\tb)
      - \Theta_{n-2}^{\mathrm{odd}}(\tb) \right),
$$
where
$$
   \Theta_n(\tb) = \prod_{i=1}^4 \frac{\sin \pi n t_i}{\sin \pi t_i},
$$
and for $n \ge 1$,
$$
\Theta_{n}^{\mathrm{odd}}(\tb) =
  \prod_{i=1}^4 \frac{\sin (n+2)\pi t_i}{\sin 2 \pi t_i}
    \sum_{j=1}^4 \frac{\sin n \pi t_j}{\sin (n+2) \pi t_j}, \quad \hbox{if $n = $ even},
$$
and
$$
 \Theta_{n}^{\mathrm{odd}}(t)  =  \prod_{i=1}^4 \frac{\sin(n+1)\pi t_i}{\sin 2 \pi t_i}
    \sum_{j=1}^4 \frac{\sin (n+3) \pi t_j}{\sin (n+1) \pi t_j}, \quad \hbox{if $n = $ odd}.
$$
\end{thm}

\begin{proof}
By definition, $\GG_n$ is a subset of $\HH_n^*$ that contains elements with
all indices being even integers. For technical reasons, it turns out to be easier
to work with $\HH_n^* \setminus \GG_n$. In fact, the sum over $\HH_n^*$ has
already been worked out in \cite{LX}, which is
$$
    \sum_{\jb \in \HH_n^*}\phi_\jb(\tb) = \prod_{i=1}^4 \frac{\sin(n+1)\pi t_i}{\sin \pi t_i}-
     \prod_{i=1}^4 \frac{\sin n \pi t_i}{\sin \pi t_i} = \Theta_{n+1}(\tb) - \Theta_n(\tb).
$$
Thus, we need to find only the sum over odd indices, that is, the sum
$$
 D_n^{\mathrm{odd}}(\tb) := \sum_{\jb \in \HH_n^{\mathrm{odd}}} \e_{\jb}(\tb),
      \qquad \HH_n^{\mathrm{odd}} := \HH_n^* \setminus \GG_n.
$$
Just as in \cite{LX}, the index set $\HH_n^{\mathrm{odd}}$ can be partitioned
into four congruent parts, each within a parallelepiped, defined by
$$
\HH_n^{(k)}: = \left\{ \jb \in \HH_n^{\mathrm{odd}}:
    0\leq j_l -j_k \leq 4 n, \,  l \in \NN_4\right\}
$$
for $k \in \NN_4$. Furthermore, for each index set $J$, $\emptyset\subset
J\subseteq  \NN_4$, define
\begin{align*}
\HH_n^{J}:= \left\{\kb \in\HH_n^{\mathrm{odd}}: k_i = k_j, \, \forall i,j \in J;
\text{ and } \, 0\leq  k_i - k_j \leq 4n,\, \forall j\in J, \ \forall i\in \NN_4\setminus J\right\}.
\end{align*}
Then we have
\begin{align*}
\HH_n^{\mathrm{odd}} =  \bigcup_{j\in \NN_4} \HH_n^{(j)} \qquad \hbox{and}\qquad
\HH_n^{J} = \bigcap_{j \in J} \HH_n^{(j)}.
\end{align*}

Using the inclusion-exclusion relation of subsets, we have
$$
   D_n^{\mathrm{odd}}(\tb) =   \sum_{\emptyset\subset J\subseteq \NN_4}
     (-1)^{|J|+1} \sum_{\kb \in \HH_n^J}  e^{\frac{\pi i}{2}\, \kb \cdot \tb}.
$$
Fix $j\in J$, using the fact that $t_j = - \sum_{i\ne j} t_i$,  we have
\begin{align*}
\sum_{\kb \in \HH_n^J} e^{\frac{\pi i}{2}\, \kb \cdot \tb}
=&  \sum_{\kb\in \HH_n^J}   e^{\frac{\pi i}{2}\, \sum_{l\in \NN_4\setminus J}
   (k_l-k_j) t_l }
= \sum_{\kb\in \HH_n^J}   \prod_{l\in \NN_4\setminus J} e^{\frac{\pi i}{2} (k_l-k_j) t_l }.
\end{align*}
Since $\kb \in \HH_n^J$ implies, in particular, $k_i \equiv k_j \mod{4}$, we obtain
\begin{align*}
\sum_{\kb \in \HH_n^J} e^{\frac{\pi i}{2}\, \kb \cdot \tb}
=   \prod_{l\in \NN_4\setminus J} \sum_{\substack{  0\leq k_l-k_j \leq 4n \\
 \kb \in \HH_n^J } }  e^{\frac{\pi i}{2} \, (k_l-k_j) t_l}
= {\prod_{l\in \NN_4\setminus J}} \sum_{\substack{ 0 \le k_l \le n \\
     |k|_J \, \mathrm{odd}} }e^{2 \pi i\, k_l t_l},
\end{align*}
where $|k|_J:=\sum_{l \in \NN_4\setminus J} k_l$. The last equation needs
a few words of explanation: if $4k_l'= k_l -k_j$, then using the fact that
$k_i =k_j$, $\forall i,j\in J$ for $k \in \HH_n^J$ and $k_1+k_2+k_3+k_4 =0$,
we see that $\frac14\sum_{l \in \NN_4\setminus J} (k_l-k_j) = - k_j$, which is odd by
the definition of $\HH_n^J$; on the other hand, assume that
$\sum_{l \in \NN_4\setminus J} k_l'$ is odd, then we define
$k_j = - \sum_{l \in \NN_4\setminus J} k_l'$ for all $j \in J$ and define
$k_l =4 kl' + k_j$, so that all components of $k$ are odd and $k \in \HH_n^J$.

The condition that $|k|_J$ is an odd integer means that the last term is
not a simple product of sums. Setting
\begin{align*}
  D_n^O(t): = & \sum_{j=0, j\,  \mathrm{odd}}^{n} \e^{  2 \pi i j t } =
     \frac{\e^{2\pi i t} (1-\e^{4 \pi i \lfloor \frac{n+1}{2} \rfloor t})}
          {1-\e^{4 \pi i t}}, \\
  D_n^E(t) := & \sum_{i=0, i\,  \mathrm{even}}^n \e^{  2 \pi i j t } =
     \frac{ 1-\e^{4 \pi i \lfloor \frac{n+2}{2} \rfloor t} }
          {1-\e^{4 \pi i t}},
\end{align*}
we see that, up to a permutation, only products $D_n^OD_n^OD_n^O$
and $D_n^OD_n^ED_n^E$ are possible for triple products ($|J|=3$), only
$D_n^OD_n^E$ is possible for double products ($|J|=2$), only $D_n^O$
is possible ($|J|=1$), and there is a constant term. Thus, using the fact
that $a b c - (a-1)(b-1)(c-1) =
a b +ac+bc - a -b-c +1$, we conclude that
\begin{align*}
& D_n^{\mathrm{odd}}(\tb) =    \sum_{(i_1,i_2,i_3)\in \NN_4}
   D_n^O(t_{i_1})D_n^O(t_{i_2})D_n^O(t_{i_3})  \\
             & + D_n^O(t_1)
                        \left[ D_n^E(t_2)D_n^E(t_3)D_n^E(t_4)
             - (D_n^E(t_2)-1)(D_n^E(t_3)-1)(D_n^E(t_4)-1) \right] \\
             & + D_n^O(t_2)
                        \left[ D_n^E(t_1)D_n^E(t_3)D_n^E(t_4)
             - (D_n^E(t_1)-1)(D_n^E(t_3)-1)(D_n^E(t_4)-1) \right] \\
             & + D_n^O(t_3)
                        \left[ D_n^E(t_1)D_n^E(t_2)D_n^E(t_4)
             - (D_n^E(t_1)-1)(D_n^E(t_2)-1)(D_n^E(t_4)-1) \right] \\
             & + D_n^O(t_4)
                        \left[ D_n^E(t_1)D_n^E(t_2)D_n^E(t_3)
             - (D_n^E(t_1)-1)(D_n^E(t_2)-1)(D_n^E(t_3)-1) \right],
\end{align*}
where the first sum is over all distinct triple integers in $\NN_4$.

Assume that $n$ is an even integer. A quick computation shows that
\begin{align*}
  D_n^O(t_1)D_n^E(t_2)D_n^E(t_2)D_n^E(t_4)
      =  \prod_{j=2}^4 \frac{\sin \pi (n+2) t_i} {\sin 2 \pi t_i}
     \frac{\sin \pi n t_1}{\sin 2 \pi t_1}.
\end{align*}
Furthermore, we see that
\begin{align*}
  & D_n^O(t_2)D_n^O(t_3)D_n^O(t_4) -
     D_n^O(t_1) (D_n^E(t_2)-1)(D_n^E(t_3)-1)(D_n^E(t_4)-1)   \\
  &= \prod_{j=2}^4 \frac{\sin \pi n t_i}{\sin 2 \pi t_i}
     \left[ \e^{i\pi n t_1} - \frac{ \e^{-2\pi i n t_1} -\e{i\pi n t_1}}{1-\e^{4 \pi i t_1}} \right]
    = - \prod_{j=2}^4 \frac{\sin \pi n t_i} {\sin 2 \pi t_i}
           \frac{\sin \pi (n-2) t_1}{\sin 2 \pi t_1}.
\end{align*}
Adding the two terms together and then summing over the permutation of the
sum, we end up the formula for $D_n^{\mathrm{odd}}(\tb)$ when $n$ is even.
The case of $n$ odd can be handled similarly.
\end{proof}

Let us write down explicitly the function $D_n(x)$ defined in \eqref{Dn3D}
in $x$-variables. Using the elementary trigonometric identity and
\eqref{coordinate1}, we see that
\begin{align*}
 4 \prod_{i=1}^4\sin \alpha \pi t_i
&  = (\cos \alpha \pi (x_2 - x_1)- \cos \alpha \pi x_3)
   (\cos \alpha \pi (x_2 + x_1)- \cos \alpha \pi x_3)  \\
& = \cos^2 \alpha x_1 + \cos^2 \alpha x_2+ \cos^2 \alpha x_3
     - 2 \cos \alpha x_1 \cos \alpha x_2 \cos \alpha x_3 -1,
\end{align*}
so that we end up with the compact formula
\begin{align} \label{Dn(x)}
   D_n(x) = \wt \Theta_{n+1}(x) - \wt \Theta_n(x) -
       \left( \wt \Theta_n^{\mathrm{odd}}(x )-  \wt \Theta_{n-2}^{\mathrm{odd}}(x)
           \right),
\end{align}
where
$$
  \wt   \Theta_n(x) = \frac{\cos^2 n \pi x_1+\cos^2 n \pi x_2+\cos^2 n \pi x_3
     -2 \cos n \pi x_1\cos n \pi x_2\cos n \pi x_3 -1}
     {\cos^2 \pi x_1+\cos^2 \pi x_2+\cos^2 \pi x_3
     -2 \cos \pi x_1\cos \pi x_2\cos \pi x_3 -1},
$$
$$
\wt \Theta_{n}^{\mathrm{odd}}(x) =
 \wt \Theta_{\frac{n+2}{2}}(2x)
    \sum_{j=1}^4 \frac{\sin n \pi t_j}{\sin (n+2) \pi t_j}, \quad \hbox{if $n = $ even}, \\
$$
and
$$
\wt  \Theta_{n}^{\mathrm{odd}}(t)  =   \wt \Theta_{\frac{n+1}{2}}(2x)
    \sum_{j=1}^4 \frac{\sin (n+3) \pi t_j}{\sin (n+1) \pi t_j}, \quad \hbox{if $n = $ odd},
$$
in which $t_i$ is given in terms of $x_j$ in \eqref{coordinate1}. As a result of
this explicit expression, we see that $D_n(x)$ is an even function in each
$x_i$.

\subsection{Boundary of the rhombic dodecahedron}
In order to develop the interpolation on the set $X_n^*$, we will need to
understand the structure of the points on the boundary of $\Lambda_n^\dag
= \ZZ^3 \cap \Omega_B$. As $\Omega_B$ is a rhombic dodecahedron, we
need to understand the boundary of this 12-face polyhedron, which has been
studied in detail in \cite{LX}. In this subsection, we state the necessary
definitions and notations on the boundary of $\Omega_B$, so that the
exposition is self-contained. We refer to further details and proofs
to \cite{LX}.

Again we use homogeneous coordinates. For $i,j \in \NN_4:=\{1,2,3,4\}$
and $i\ne j$, the (closed) faces of $\Omega_B$ are
$$
         F_{i,j}  = \{ \tb \in \overline{ \Omega}_H: t_i - t_j =1\}.
$$
There are a total $2 \binom{4}{2}  = 12$ distinct $F_{i,j}$, each represents
one face of the rhombic dodecahedron. For nonempty subsets $I, J$ of $\NN_4$,
define
\begin{align*}
  \Omega_{I,J} :=  \bigcap_{i\in I, j\in J} F_{i,j} =
     \left\{ \tb \in \overline{ \Omega}_H:\  t_j = t_i-1, \text{ for all } i \in I,  j\in J\right\}.
\end{align*}
It is shown in \cite{LX} that $\Omega_{I,J} = \emptyset$ if and only if $I\cap J \neq \emptyset$, and  $\Omega_{I_1,J_1} \cap \Omega_{I_2,J_2} =\Omega_{I, J}$  if
  $I_1 \cup I_2=I$ and $J_1\cup J_2=J$.
These sets describe the intersections of faces, which can then be used to
describe the edges, which are intersections of faces, and vertices, which
are intersections of edges. Let
\begin{align*}
  & \CK := \left\{(I,J):  I, J \subset \NN_4; \   I\cap J = \emptyset \right\},\\
  & \CK_0 := \left\{ (I,J)\in \CK:\ i<j, \,\, \hbox{for all}\,\, (i,j) \in (I, J)\
 \right\}.
\end{align*}
We now define, for each $(I,J)\in \CK$, the boundary element $\B_{I,J}$
of the dodecahedron,
\begin{align*}
  \B_{I,J}: = \left\{\tb \in \Omega_{I,J}: \  \tb \not \in \Omega_{I_1,J_1}
\text{ for all }
  (I_1,J_1)\in \CK \text{ with } |I|+|J| < |I_1| + |J_1| \right\};
\end{align*}
it is called a face if $|I|+|J| =2$, an edge if $|I|+|J|=3$, and a vertex if $|I|+
|J|=4$. By definition, the elements for faces and edges are without boundary,
which implies that $\B_{I,J} \cap \B_{I',J'} =  \emptyset$ if $I \neq I_1$ and
$J \neq J_1$. In particular, it follows that $\B_{\{i\},\{j\}} = F_{i,j}^\circ$ and,
for example, $\B_{\{i\}, \{j,k\}} = (F_{i,j} \cap F_{i,k})^\circ$ for distinct integers
$i, j, k  \in \NN_4$.

Let  $\mathcal{G}=S_4$ denote the permutation group of four elements
and let $\sigma_{ij}$ denote the element in $\CG$ that interchanges $i$
and $j$; then $\tb \sigma_{ij} = \tb - (t_i-t_j) \e_{i,j}$. For a nonempty set
$I\subset \mathbb{N}_4$, define $\mathcal{G}_{I} := \left\{\sigma_{ij}:  i,j\in I\right\}$,
where we take $\sigma_{ij} = \sigma_{ji}$ and take $\sigma_{jj}$ as the identity
element. It follows that $\mathcal{G}_{I} $ forms a subgroup of $\mathcal{G}$
of order $|I|$. For $(I, J)\in \mathcal{K}$, we then define
\begin{align} \label{eq:[B]}
   [\B_{I,J}] := \bigcup_{\sigma\in \mathcal{G}_{I\cup J}} \B_{I,J}\sigma.
\end{align}
It turns out that $ [\B_{I,J}] $ consists of exactly those boundary elements that
can be obtained from $\B_{I,J}$ by congruent modulus $B$, and
$[\B_{I,J}]\cap [\B_{I_1,J_1}] = \emptyset$ if $(I,J) \neq (I_1,J_1)$ for $(I,J)\in
 \mathcal{K}_0$ and $(I_1,J_1)\in \mathcal{K}_0$. More importantly,
 we define, for $0 < i,j < i+j \le 4$,
\begin{align} \label{eq:Bij=[B]}
\begin{split}
  \B^{i,j}:= \bigcup_{(I,J) \in \CK_0^{i,j}} [\B_{I,J}]
\quad \hbox{with}\quad
   \CK^{i,j}_0: = &\, \left\{(I,J) \in \CK_0:\  |I| = i,\ |J| =j  \right\}.
\end{split}
\end{align}
Then the boundary of $\overline{\Omega}_B$ can be decomposed as
$$
   \overline{\Omega}_H\setminus \Omega_H^{\circ}
 =  \bigcup_{(I,J)\in \CK}  \B_{I,J}  = \bigcup_{0<i,j<i+j\le 4} \B^{i,j}.
$$

The main complication is the case of $|I|+|J|=2$, for which
we have, for example,
\begin{align} \label{[B{1,23}]}
  [\B_{\{1\},\{2,3\}}] = \B_{\{1\},\{2,3\}} \cup \B_{\{2\},\{1,3\}} \cup \B_{\{3\},\{1,2\}}.
\end{align}
The other cases can be written down similarly. Furthermore, we have
\begin{align} \label{eq:Bij}
\begin{split}
& \B_{\{1\},\{2,4\}} = \B_{\{1\},\{2,3\}} \sigma_{34}, \qquad \quad\,
\B_{\{1,2\},\{4\}} = \B_{\{1,2\},\{3\}} \sigma_{34}, \\
& \B_{\{1\},\{3,4\}} =\B_{\{1\},\{2,3\}} \sigma_{24},
\quad B_{\{1,3\},\{4\}}= \B_{\{1,2\},\{3\}} \sigma_{23}\sigma_{34}, \\
& \B_{\{2\},\{3,4\}} = \B_{\{1,2\},\{3\}} \sigma_{12} \sigma_{24}, \quad \,\,
\B_{\{2,3\},\{4\}} = \B_{\{1,2\},\{3\}} \sigma_{13}\sigma_{34},
\end{split}
\end{align}
with
\begin{align} \label{B_{1,{2,3}}}
\begin{split}
& \B_{\{1\},\{2,3\}} = \left \{(t, t-1, t-1, 2-3t):  \tfrac12 < t <\tfrac34  \right \}, \\
& \B_{\{1,2\},\{3\}}= \left\{ (1-t, 1-t, -t, 3t-2):  \   \tfrac12 < t <\tfrac34  \right\}.
\end{split}
\end{align}

If $|I|+|J| =2$ then $\B_{I,J} = \B_{\{i\},\{j\}}$ is a face and
\begin{align*} 
\B^{1,1}   = [\B_{\{1\},\{2\}}] \cup  [\B_{\{1\},\{3\}}]\cup [\B_{\{1\},\{4\}}]\cup
   [\B_{\{2\},\{3\}}]\cup[\B_{\{2\},\{4\}}]\cup   [\B_{\{3\},\{4\}}]
\end{align*}
If $|I|+|J| =3$ then
$\B_{I,J}$ is an edge and we have
\begin{align} \label{eq:B12}
\begin{split}
 &\B^{1,2}   = [\B_{\{1\},\{2,3\}}] \cup  [\B_{\{1\},\{2,4\}}]\cup [\B_{\{1\},\{3,4\}}]\cup
   [\B_{\{2\},\{3,4\}}],\\
& \B^{2,1} =  [\B_{\{1,2\},\{3\}}]\cup  [\B_{\{1,2\},\{4\}}]\cup [\B_{\{1,3\},\{4\} }]
  \cup [\B_{\{2,3\},\{4\}}].
\end{split}
\end{align}
If $|I|+|J|=4$, then
\begin{align} \label{eq:B13}
\begin{split}
& \B^{1,3} = \left[\{(\tfrac{1}{4}, \tfrac{1}{4}, \tfrac{1}{4},
 -\tfrac{3}{4})\}\right], \quad \B^{2,2} = \left[\{(\tfrac{1}{2}, \tfrac{1}{2}, -\tfrac{1}{2},
 -\tfrac{1}{2}\})\right]\\
& \B^{3,1} = \left[\{(\tfrac{3}{4}, -\tfrac{1}{4}, -\tfrac{1}{4},  -\tfrac{1}{4})\}\right].
\end{split}
\end{align}

Recall that $\GG_n$ is $\Lambda_n^{\dag*} = \ZZ^3 \cap \overline{\Omega}_B$
in homogeneous coordinates.  We now consider the decomposition of the
boundary of $\GG_n$ according to the boundary elements of the rhombic
dodecahedron. First we denote by $\GG_n^\circ$ the points inside $\GG_n$,
\begin{align*}
  \GG_n^{\circ}:= \left\{ \jb \in \GG_n: -4n < j_\nu -j_\mu<4n, 1\le \nu,\mu \le 4\right\}
     = \left\{ \jb \in \GG_n: \tfrac{\jb}{4n} \in \Omega_B^\circ \right\}.
\end{align*}
We further define, for $0 < i, j < i+j \leq 4$,
\begin{align} \label{eq:Hn^ij}
  \GG_n^{i,j} := \left\{ \kb \in \GG_n:  \tfrac{\kb}{4n} \in \B^{i,j}  \right\}
 \end{align}
The set $\GG_n^{i,j}$ describes those points $\jb$ in $\GG_n$ such that
$\frac{\jb}{4n}$ are in $B^{i,j}$ of $\partial \Omega_B$. It is easy to see that
$\GG_n^{i,j} \cap \GG_n^{k,l} = \emptyset$ if $i\ne k, j \ne l$ and
\begin{align*}
\bigcup_{0<i,j<i+j\leq 4} \GG_n^{i,j}   = \GG_n \setminus \GG_n^{\circ}.
\end{align*}

\subsection{Interpolation by trigonometric polynomials}
We first apply the general theory from Section 2 to our set up with
$\Omega_B$ as a rhombic dodecahedron.

\begin{thm} \label{thm:interp3D}
For $n \ge 1$ define
\begin{equation}\label{CI_n3d}
    I_n f(x) : = \sum_{ k \in X_n} f(\tfrac{k}{2n}) \Phi_n(x- \tfrac{k}{2n}),
      \qquad
         \Phi_n(x) := \frac{1}{2n^3} \sum_{\nu \in \Lambda_n^\dag} \e_\nu(x).
\end{equation}
Then for each $j \in X_n$, $I_n(\tfrac{j}{2n}) = f(\tfrac{j}{2n})$.
\end{thm}

\begin{proof}
By \eqref{j-k3}, $I_n f(\frac{j}{2n}) = f (\frac{j}{2n})$ for $j \in X_n$ is equivalent
to $I_n f(B^{-\tr}l ) = f(B^{-\tr} l)$ for $l \in \Lambda_n$. Moreover, $I_n f$ can
be rewritten as
$$
    I_n f(x)  = \sum_{ j \in \Lambda_n} f(B^{-\tr} j) \Phi_n(x- B^{-\tr} j).
$$
Hence, this theorem is a special case of Theorem \ref{thm:interpolation}.
\end{proof}


Next we consider interpolation on the symmetric set of points $X_n^*$. For
this we need to modify the kernel function $\Phi_n$. Recall that, under the
change of variables \eqref{eq:j-k}, $\Lambda_n^{\dag *}$ becomes
$\GG_n$ in homogeneous coordinates. We define
$$
   \Phi_n^*(x) :=  \frac{1}{2n^3} \sum_{\nu \in \Lambda_n^{\dag *}}
            \wt \mu_\nu^{(n)}  \e_\nu(x)
             = \frac{1}{2n^3} \sum_{\jb \in \GG_n}
            \mu_\jb^{(n)}  \e_\jb(\tb),
$$
where $x$ and $\tb$ are related by \eqref{coordinate1}, $\wt \mu_k^{(n)}$
is defined by $\mu_k^{(n)}$ under the change of indices \eqref{eq:j-k}, and
$\mu_{\jb}^{(n)}=1$ if $ \jb \in \GG_n^{\circ}$,
$\mu_{\jb}^{(n)}=\frac{1}{\binom{i+j}{i}}$ if $ \jb \in \GG_n^{i,j}$;  more
explicitly
\begin{align*}
    \mu_{\jb}^{(n)} :=\begin{cases} 1, & \jb \in \GG_n^{\circ} \\
           \frac12, & \jb \in \GG_n^{1,1}, \\
       \frac13, & \jb \in \GG_n^{1,2}\cup \GG_n^{2,1}, \\
       \frac14, & \jb \in \GG_n^{1,3}\cup \GG_n^{3,1}, \\
       \frac16, & \jb \in \GG_n^{2,2}.
\end{cases}
\end{align*}

For each $k$ on the boundary of $X_n^*$, that is, $\frac{k}{2n}$ on the boundary
of $[-\frac12, \frac12]^3$, let
\begin{equation}\label{CSk}
  \CS_k : =\{j \in X_n^*:  \tfrac{j}{2n} \equiv  \tfrac{k}{2n} \mod{\ZZ^3} \},
\end{equation}
which contains the points on the boundary of $X_n^*$ that are congruent to $k$
under integer translations.

\begin{thm} \label{thm:3dInterp}
For $n \ge 1$ define
\begin{equation}\label{I_n3d}
    I_n^* f(x) : = \sum_{ k \in X_n^*} f(\tfrac{k}{2n}) R_k(x), \qquad
      R_k(x) :=   \Phi_n^*(x- \tfrac{k}{2n}).
\end{equation}
Then for each $j \in X_n^*$,
\begin{align}    \label{eq:sym-interp}
        I^*_n f(\tfrac{j}{2n}) = \begin{cases}
         f(\frac{j}{2n}), & j \in X^{\circ}_n,\\
                  \\
         \displaystyle \sum_{k \in S_{j}}  f(\tfrac{k}{2n}), & j \in X_n^*
                 \setminus X_n^{\circ}.
\end{cases}
\end{align}
In homogeneous coordinates, the function $\Phi_n^*(x) =\wt \Phi_n^*(\tb)$ is
a real function and it satisfies
\begin{align} \label{eq:Phi-n*}
 \wt \Phi_n^*(\tb) = & \frac{1}{4n^3} \left[ \frac{1}{2}
         \left( D_n^*(\tb) + D_{n-1}^*(\tb) \right) -  \frac{1}{3}
         \sum_{\nu=1}^4  \frac{\sin 2 \pi  \lfloor \frac{n-1}{2} \rfloor  t_\nu}
            {\sin 2 \pi t_\nu}    \sum_{\substack{j=1 \\ j \ne \nu}}^4
                \cos 2 \pi (n t_j + \lfloor \tfrac{n}{2} \rfloor t_\nu) \right. \notag \\
    & \qquad \left. \qquad  -
       \frac13   \sum_{1 \le \mu<\nu \le 4}  \cos 2\pi  n (t_{\mu} + t_{\nu})
       - \frac{1}{2} \begin{cases} \sum_{j=1}^4 \cos 2 \pi n t_j, & \hbox{if $n$ even} \\
         0  & \hbox{if $n$ odd} \end{cases} \right],
\end{align}
from which the formula for $\Phi_n^*(x)$ follows from \eqref{coordinate1}
and \eqref{Dn(x)}.
\end{thm}

\begin{proof}
By \eqref{j-k3}, we need to verify the interpolation at the points
$B^{-\tr}l$ for $l \in \Lambda_n^*$. By definition, we can write
$$
  R_k (B^{-\tr} l) =       \frac{1}{2n^3} \sum_{\nu \in \Lambda_n^{\dag *}}
           \wt \mu_\nu^{(n)}  \e_\nu(B^{-\tr}(l-k) ).
$$
It is easy to see that $\nu^\tr B^{-\tr}l =\frac{1}{4n}( j_1l_1+j_2l_2+j_3l_3)$
if $\nu$ is related to $\jb$ by \eqref{eq:j-k}. Hence, as in the proof of
Theorem 3.15 in \cite{LX}, we conclude that
$$
 R_k (B^{-\tr} l)  =
       \frac{1}{2n^3} \sum_{\nu \in \Lambda_n^{\dag}} \e_\nu(B^{-\tr}(l-k) ),
$$
Now,  for $l, k \in \Lambda_n^*$, there exist $p \in \Lambda_n$ and $q \in \ZZ^3$
such that $l-k \equiv p \pm B^\tr q$. Consequently, it follows
from \eqref{d-ortho2} that
$$
R_k(B^{-\tr} l)= \frac{1}{2n^3} \sum_{\nu \in \Lambda_n^{\dag}} \e_\nu(B^{-\tr} p)
  = \delta_{p,0}.
$$
By \eqref{j-k3}, we have verified that
\begin{equation} \label{eq:Rk3D}
   R_k (\tfrac{j}{2n}) =  \begin{cases} 1, & \frac{j}{2n} \equiv \frac{k}{2n}
     \mod \ZZ^3,\\
                 0, & \hbox{otherwise},
\end{cases}
\end{equation}
which proves the interpolation part of the theorem.

In order to prove the compact formula, we start with the following formula
that can be established exactly as in the proof of Theorem 3.15 in \cite{LX}:
\begin{align} \label{eq:Phi*}
  \wt \Phi^*_n(\tb) = \frac{1}{4n^3}\bigg[&\frac12(D_n^*(\tb)+D_{n-1}^*(\tb))
   - \frac16 \sum_{k\in \GG_n^{1,2}\cup \GG_n^{2,1}} \phi_\kb(\tb)\\
   &- \frac14 \sum_{k\in \GG_n^{1,3}\cup \GG_n^{3,1}} \phi_\kb(\tb)
   - \frac1{3} \sum_{k\in \GG_n^{2,2} } \phi_\kb(\tb)
    \bigg].\notag
\end{align}
Let us define $\GG_n^{I,J}: = \{\kb \in \GG_n: \tfrac{\kb}{4n} \in \B_{I,J}\}$
for $I, J \subset \NN_4$ and also define $\left[\GG_n^{I,J}\right]: =
\{\kb \in \GG_n: \tfrac{\kb}{4n} \in [\B_{I,J}]\}$. It follows from \eqref{eq:Bij=[B]},
and \eqref{eq:Hn^ij} that
$$
\GG_n^{i,j} = \bigcup_{I,J\in \CK_0^{i,j}} \left[\GG_n^{I,J} \right]
\qquad\hbox{and}\qquad
  \left[\GG_n^{I,J} \right] =  \bigcup_{\sigma\in \mathcal{G}_{I\cup J}}
     \GG_n^{I,J}\sigma.
$$
In order to compute the sums in \eqref{eq:Phi*}, we need to use the
detail description of the boundary elements of $\Omega_B$ in the
previous subsection. The computation is parallel to the proof of
Theorem 3.15 in \cite{LX}, in which the similar computation with
$\GG_n$ replaced by $\HH_n$ is carried out. Thus, we shall be brief.

Using $t_1+t_2+t_3+t_4 =0$ and the explicit description of
$\B^{\{1\},\{2,3\}}$, we get
\begin{align*}
 &  \sum_{\kb \in [\GG_n^{\{1\},\{2,3\}}]} \phi_\kb(\tb)   =
       \sum_{\kb \in \GG_n^{\{1\},\{2,3\}}} \e^{\frac{\pi i}{2}\kb\cdot \tb} +
       \sum_{\kb \in \GG_n^{\{2\},\{1,3\}}} \e^{\frac{\pi i}{2}\kb\cdot \tb} +
        \sum_{\kb \in \GG_n^{\{3\},\{1,2\}}} \e^{\frac{\pi i}{2}\kb\cdot \tb} \\
    & \qquad =    \sum_{j=1, j \mathrm{even}}^{n-1} e^{-2 \pi i j t_4}  \left(
      e^{2 n \pi i (t_1+t_4)} + e^{2 n \pi i (t_2+t_4)} +e^{2 n \pi i (t_3+t_4)} \right) \\
    & \qquad
    = \frac{\sin 2 \pi \lfloor \frac{n-1}{2} \rfloor t_4}{\sin 2 \pi t_4}
   \e^{- 2\pi i \lfloor \frac{n+1}{2} \rfloor  t_4} \left(e^{2  \pi i n ( t_1+t_4)} +
          e^{2 \pi i n( t_2+t_4)}+ e^{2 \pi i n( t_3+t_4)} \right),
\end{align*}
Similarly, we also have
\begin{align*}
   &  \sum_{\kb \in [\GG_n^{\{1,2\},\{3\}}]} \phi_\kb(\tb)   =
       \sum_{\kb \in \GG_n^{\{1\},\{2,3\}}} \e^{\frac{\pi i}{2}\kb\cdot \tb} +
       \sum_{\kb \in \GG_n^{\{2\},\{1,3\}}} \e^{\frac{\pi i}{2}\kb\cdot \tb} +
        \sum_{\kb \in \GG_n^{\{3\},\{1,2\}}} \e^{\frac{\pi i}{2}\kb\cdot \tb} \\
   & \qquad
    = \frac{\sin 2 \pi \lfloor \frac{n-1}{2} \rfloor t_4}{\sin 2 \pi t_4}
   \e^{ 2\pi i \lfloor \frac{n+1}{2} \rfloor  t_4} \left(e^{-2 \pi i n ( t_1+t_4)} +
          e^{-2 \pi i n( t_2+t_4)}+ e^{-2 \pi i n( t_3+t_4)} \right).
\end{align*}
From these and their permutations, we can compute the sum over
$\GG_n^{1,2}$ and $\GG_n^{2,1}$. Putting them together, we obtain
\begin{align*}
   \sum_{\kb\in \GG_n^{1,2}\cup \GG_n^{2,1}} \phi_k(\tb)  =
       2 \sum_{\nu=1}^4  \frac{\sin 2 \pi  \lfloor \frac{n-1}{2} \rfloor  t_\nu}
            {\sin 2 \pi t_\nu}    \sum_{\substack{j=1 \\ j \ne \nu}}^4
                \cos 2 \pi (n t_j + \lfloor \tfrac{n}{2} \rfloor t_\nu).
\end{align*}
Using \eqref{eq:B13}, we see that, $\GG_n^{2,2} = \{(2n,2n,-2n,-2n)\sigma:
\sigma \in \CG\}$ and, if $n$ is even then
$\GG_n^{1,3} = \{(n,n,n,-3n)\sigma: \sigma \in \CG\}$ and
$\GG_n^{3,1} = \{(3n,-n,-n,-n)\sigma: \sigma \in \CG\}$, whereas if $n$ is odd,
then $\GG_n^{1.3} = \GG_n^{3,1} =\emptyset$. As a result, it follows that
\begin{align*}
 \sum_{\kb\in \GG^{2,2}_n} \phi_\kb(\tb)
  =  \sum_{1 \le \mu<\nu \le 4}  e^{2\pi i n (t_{\mu} + t_{\nu})}
  =   \sum_{1 \le \mu<\nu \le 4}  \cos 2\pi  n (t_{\mu} + t_{\nu}),
\end{align*}
where we have used the fact that $t_1+t_2+t_3+t_4 =0$, and
\begin{align*}
   \sum_{\kb\in \GG^{1,3}_n\cup \GG^{3,1}_n} \phi_k(\tb) =  \sum_{j=1}^4
   \big( e^{2\pi i n t_j} + e^{-2\pi i n t_j} \big)
   =  2 \sum_{j=1}^4 \cos 2 \pi  n t_j,
\end{align*}
if $n$ is even, whereas it  is equal to $0$ if $n$ is odd.

Putting all these into \eqref{eq:Phi*} completes the proof.
\end{proof}

\begin{thm}\label{LebesgueH}
Let $\|I_n^*\|_\infty$ denote the norm of the operator $I_n^*: C([-\tfrac12, \tfrac12]^3)
 \mapsto C([-\tfrac12, \tfrac12]^3)$. Then there is a constant $c$, independent of
$n$, such that
$$
       \|I_n^*\|_\infty  \le c (\log n)^3.
$$
\end{thm}

\begin{proof}
Following the standard procedure, we see that
$$
   \|I_n^*\|_\infty = \max_{x \in [-\frac12,\frac12]^3}   \sum_{k \in X_n^*}
         \left|\Phi_n^*(x -\tfrac{k}{4n})\right|.
$$
Using the formula of $\Phi_n^*$ in \eqref{eq:Phi-n*}, it is
easy to see that it suffices to prove that
$$
   \max_{x\in [-\frac12,\frac12]^3}  \sum_{k \in X_n^*}
         \left|D_n^*(x -\tfrac{k}{2n})\right| \le c (\log n)^3, \qquad n \ge 0.
$$
Furthermore, using the explicit formula of $D_n^{\mathrm{odd}}(\tb)$ and
(3.19) in \cite{LX}, we see that our main task is to estimate the sums in the
form of
$$
I_{\{1,2,3\}}:= \frac{1}{2n^3} \max_{\tb \in Q}   \sum_{k \in X_n^*}
   \left| \frac{\sin \pi n (t_1- \frac{k_2+k_3}{2n})\sin \pi n (t_2- \frac{k_1+k_3}{2n})
   \sin \pi n (t_3- \frac{k_1+k_2}{2n})}{\sin \pi (t_1- \frac{k_2+k_3}{2n})
      \sin \pi (t_2- \frac{k_1+k_3}{2n})
   \sin \pi (t_3- \frac{k_1+k_2}{2n})} \right|
$$
and three other similar estimates $I_{\{1,2,4\}}$, $I_{\{1,3,4\}}$ and $I_{\{2,3,4\}}$, respectively, as well as similar sums in which the denominator becomes
product of $\sin2 \pi (t_i-\frac{k_i}{2n})$ and $n$ in the numerator is replace
by $n+1$ or $n+2$. Here $Q$ is the image of $[-1,1]^3$ under the mapping
\eqref{coordinate1}; that is,
$$
Q= \{\tb \in \RR_H^4: -\tfrac12 \le t_1+t_2, t_2+t_3, t_3 + t_1 \le \tfrac12\}.
$$
Changing the summation indices and enlarging the set $X_n^*$, we see that
\begin{align*}
   I_{\{1,2,3\}} & \le  4  \max_{t \in {[-1,1]}} \left (\frac{1}{2n}
        \sum_{k = 0}^{2 n} \left | \frac{\sin n \pi (t -\tfrac{k}{2 n})}
          {\sin \pi (t -\tfrac{k}{2 n})} \right|  \right)^3   \le c (\log n)^3,
\end{align*}
where the last step follows from the standard estimate of one variable
(cf. \cite[Vol. II, p. 19]{Z}).
\end{proof}

\subsection{Interpolation by algebraic polynomials} The main outcome
of Theorem \ref{thm:interp3D} in the previous section is that we can
derive a genuine interpolation by trigonometric polynomials based
on the set of points in $\{\frac{k}{2n}:k\in \Xi_n\}$ defined at \eqref{Xi_n}.
The development below is similar to the case of $d =2$. We define
$$
   \CP f(x):= \frac{1}{8}
        \sum_{\eps \in \{-1,1\}^3} f(\eps_1 x_1, \eps_2 x_2,  \eps_3 x_3).
$$

\begin{thm}
For $n \ge 0$ define
$$
  \CL_n f(x) = \sum_{k \in \Xi_n} f(\tfrac{k}{2n}) \ell_k(x), \qquad
      \ell_k(x): = \lambda_k^{(n)} \CP \left[\Phi_n^*(\cdot - \tfrac{k}{2n})\right](x)
$$
with $\l_k^{(n)}$ given in \eqref{cubaT3d}. Then $\CL_n f \in \CT_n$ is even
in each of its variables and it satisfies
$$
  \CL_n f(\tfrac{j}{2n}) = f(\tfrac{j}{2n})  \qquad \hbox{for all} \quad j \in \Xi_n.
$$
\end{thm}

\begin{proof}
As shown in \eqref{eq:Rk3D}, $R_k(x):=
\Phi_n^*(x- \tfrac{k}{2n})$ satisfies $R_k(\tfrac{j}{2n}) = 1$  when
$k \equiv j \mod 2n \ZZ^3$ and 0 otherwise. Hence, if $j \in \Xi_n^\circ$ then
$(\CP R_k)(\tfrac{j}{2n}) = \frac{1}{8} R_k(\tfrac{j}{2n}) = [\l_k^{(n)}]^{-1} \d_{k,j}$.
If $j \in \Xi_n^*\setminus \Xi_n^\circ$, then we need to consider several
cases, depending on how many components of $\jb$ are zero, which
determines how many distinct terms are in the sum  $(\CP R_k)(\tfrac{j}{2n})$
and how many distinct $k$ can be obtained from $j$ by congruent
in $\ZZ^3$.  For example, if $j \in \Xi_n^f$ and none of the components of
$j$ are zero, then there are 2 elements in $\CS_j$, $j$ and the one in
the opposite face, and the sum $\CP R_k(\tfrac{j}{2n})$ contains 8 terms,
so that $(\CP R_k)(\tfrac{j}{2n}) = \frac{1}{4} \delta_{j,k} = [\lambda_k^{(n)}]^{-1}
\delta_{k,j}$. The other cases can be verified similarly, just as in the case of
$d=2$. We omit the details.
\end{proof}

The above theorem yields immediately interpolation by algebraic polynomials
upon applying the change of variables \eqref{t-x3d}. Recall
$\Gamma_n$ defined in \eqref{Gamma_n3d} and the polynomial subspace
$$
   \Pi_n^*= \mathrm{span} \{s_1^{k_1}s_2^{k_2}s_3^{k_3}: k_1,k_2,k_3 \ge 0,
         k_i+k_j \le n, 1 \le i,j \le 3\}.
$$

\begin{thm}
For $n \ge 0$, let
$$
   \CL_n f (s) = \sum_{z_k \in \Gamma_n} f(z_k) \ell^*_k(s),
         \qquad
         \ell_k^*(s) = \ell_k(x) \quad\hbox{with}\quad  s= \cos 2 \pi x.
$$
Then $\CL_n f \in \Pi_n^*$ and it satisfies $\CL_n f(z_k) = f(z_k)$ for all
$z_k \in \Gamma_n$.
\end{thm}

This theorem follows immediately from the change of variables
\eqref{t-x3d}. The explicit compact formula of $\ell_k(x)$, thus $\ell^*_k(s)$,
can be derived from Theorem \ref{thm:3dInterp}.

The theorem states that the interpolation space for the point set $\Gamma_n$
is exactly $\Pi_n^*$, which consists of monomials that have indices in
the positive quadrant of the rhombic dodecahedron, as depicted in
Figure 3 below.

\begin{figure}[ht]
\centering
\includegraphics[width=0.5\textwidth]{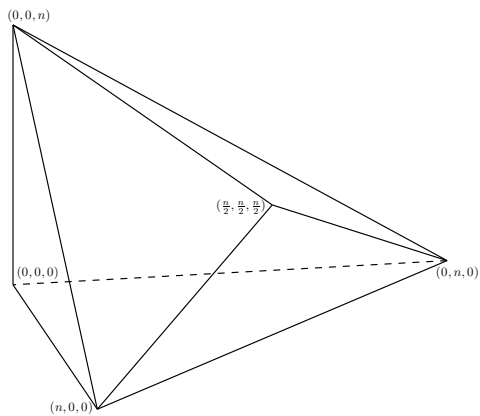}
\caption{Index set of $\Pi_n^*$}
\end{figure}

\noindent
The set $\Gamma_n$ consists of roughly  $n^3/4(1+ \CO(n^{-1})$
points. The interpolation polynomial $\CL_n f \in \Pi_n^*$ is about a total degree
of $3 n/2$. The compact formula of the fundamental interpolation polynomial
provides a convenient way of evaluating the interpolation polynomial. Furthermore,
the Lebesgue constant of this interpolation process remains at the order of
$(\log n)^3$, as the consequence of Theorem \ref{LebesgueH} and the
change of variables.

\begin{cor}\label{Lebesgue}
Let $\| \CL_n\|_\infty$ denote the operator norm of $\CL_n: C([-1, 1]^3)
 \mapsto C([-1, 1]^3)$. Then there is a constant $c$, independent of
$n$, such that
$$
       \|\CL_n\|_\infty  \le c (\log n)^3.
$$
\end{cor}

 \end{document}